\documentclass{article}

\usepackage{enumitem}
\RequirePackage{amsthm,amsmath,amsfonts,amssymb}
\RequirePackage[numbers]{natbib}
\RequirePackage[colorlinks,citecolor=blue,urlcolor=blue]{hyperref}
\RequirePackage{graphicx}
\usepackage[a4paper, margin=.8in]{geometry}

\usepackage[utf8]{inputenc}
\usepackage{amsmath}
\usepackage{bbm}
\usepackage{amsthm}
\usepackage{amssymb}
\usepackage{dsfont}

\usepackage{xcolor}

\usepackage{placeins}

\usepackage{hyperref}

\usepackage{mathrsfs} 
\usepackage{longtable}
\usepackage{authblk}

\usepackage[algo2e]{algorithm2e}

\makeatletter
\newcommand*{\addFileDependency}[1]{
  \typeout{(#1)}
  \@addtofilelist{#1}
  \IfFileExists{#1}{}{\typeout{No file #1.}}
}
\makeatother

\newcommand{\va}{\text{Var}} 
\newcommand{\ex}{\text{E}} 
\newcommand{\E}{\ex}
\newcommand{\R}{\mathbb{R}} 
\newcommand{\N}{\mathbb{N}} 
\newcommand{\abs}[1]{|#1|} 
\newcommand{\cI}{\mathcal{I}} 
\newcommand{\1}{\mathbf{1}} 
\newcommand{\bx}{\mathbf{x}}

\newcommand{\leaf}{{\text{leaf}}} 

\newcommand{\mt}{m_{\text{try}}} 
\newcommand{\Mt}{M_{\text{try}}} 

\newcommand{\cP}{\mathcal{P}} 
\newcommand{\cp}{\mathfrak{p}} 
\newcommand{\cB}{\mathcal{B}} 
\newcommand{\cF}{\mathcal{F}} 
\newcommand{\cD}{\mathcal{D}} 

\newcommand{\pto}{\overset{p}{\to}}
\newcommand{\parent}{\mathrm{F}}

\newtheorem{theorem}{Theorem}
\newtheorem*{model}{LSS model}

\newtheorem{lemma}{Lemma}
\newtheorem{proposition}[lemma]{Proposition}
\newtheorem{definition}{Definition}
\newtheorem{assumption}{}

\begin{document}

\author[1]{Merle Behr\footnote{Equal contribution.}}
\author[1]{Yu Wang\footnote{Equal contribution.}}
\author[1]{Xiao Li}
\author[1, 2, 3]{Bin Yu\footnote{To whom correspondence should be addressed. E-mail: binyu@berkeley.edu}}
\affil[1]{Department of Statistics, UC Berkeley}
\affil[2]{Department of Electrical Engineering and Computer Sciences, UC Berkeley}
\affil[3]{Center for Computational Biology, UC Berkeley}

\title{Provable Boolean Interaction Recovery from Tree Ensemble obtained via Random Forests}
\date{\today}

\maketitle
\begin{abstract}
Random Forests (RF) are at the cutting edge of supervised machine learning in terms of prediction performance, especially in genomics.  
Iterative Random Forests (iRF)
use a tree ensemble from iteratively modified RF to obtain predictive and stable non-linear or Boolean interactions of features.
They have shown great promise for Boolean biological interaction discovery that is central to advancing functional genomics and precision medicine.
However, theoretical studies into how tree-based methods discover Boolean feature interactions are missing.
Inspired by the thresholding behavior in many biological processes, we first introduce a novel discontinuous nonlinear regression model, called the \emph{Locally Spiky Sparse (LSS)} model.
Specifically, the LSS model assumes that the regression function is a linear combination of piecewise constant Boolean interaction terms.
Given an RF tree ensemble, we define a quantity called \emph{Depth-Weighted Prevalence (DWP)} for a set of signed features $S^\pm$. Intuitively speaking, DWP($S^\pm$) measures how frequently features in $S^\pm$ appear together in an RF tree ensemble.
We prove that, with high probability, DWP($S^\pm$) attains a universal upper bound that does not involve any model coefficients, if and only if $S^\pm$ corresponds to a union of Boolean interactions under the LSS model.
Consequentially, we show that a theoretically tractable version of the iRF procedure, called LSSFind, yields consistent interaction discovery under the LSS model as the sample size goes to infinity.
Finally, simulation results show that LSSFind recovers the interactions under the LSS model even when some assumptions are violated.
\end{abstract}

\section{Introduction}

Supervised machine learning algorithms have been proven to be extremely powerful in a wide range of predictive tasks from genomics, to cosmology, to pharmacology.
Understanding how a model makes predictions is of paramount value in science and business alike \cite{Yu2020}. 
For example, when a geneticist wants to understand a particular disease, e.g. breast cancer, a black-box algorithm \textit{predicting} the risk of breast cancer from genotype features is useful, but it does not offer \textit{biological insight}. 
That is, discovery of genes and gene interactions driving a particular disease provides not only understanding as a basic goal in science, but also opens doors for therapeutic treatments.
It is a pressing task, in genomics and beyond, to interpret supervised machine learning (ML) models or algorithms and extract mechanistic information in addition to prediction.

Among many supervised ML algorithms, tree ensembles such as those from Random Forests (RF) \cite{Breiman2001} and gradient boosted decision trees \cite{friedman2001greedy} stand out as they enjoy both state-of-the-art prediction performance in a variety of practical problems and lead to relatively simple interpretations \cite{strobl2007unbiased,louppe2013,zhou2020a,loecher2020unbiased,li2019}. 
To interpret a tree ensemble model, two questions are central:
\begin{itemize}
\item \textbf{Feature importance}: What \textit{features} are important for the model's prediction? 
\item \textbf{Interaction importance}: What \textit{interactions among features} are  important for the model's prediction?
\end{itemize}

While many studies (see \cite{strobl2007unbiased,zhou2020a,li2019,loecher2020unbiased} and the references therein) focus on the RF feature importance, there are relatively few results on the second question. 
In genetics, Wan et al. and Yoshida and Koike  \cite{wan2009,yoshida2011} seek (higher-order) gene-interactions (or epistasis) by  extracting genetic variant interactions from paths of ensembles of fitted decision trees. 
Wan et al. \cite{wan2009} use MegaSNPHunter based on boosting trees and interpret all groups of features that jointly appear on one of the decision paths as a candidate interaction.
Yoshida and Koike \cite{yoshida2011} propose to rank interactions of genetic variants based on how often they appear together on decision paths in an RF tree ensemble.
Recently, iterative Random Forests (iRF) \cite{Basu2017} is proposed to seek predictive, stable, and high-order non-linear or Boolean feature interactions. 
Even though iRF uses the idea that the set of interacting features often appear together on individual decision paths of a tree in an RF ensemble as in Yoshida and Koike \cite{yoshida2011}, it uses several other ideas. That is, iRF incorporates a soft dimension reduction step via iterative re-weighting of features in terms of their Gini importance, in order to stabilize individual decision paths in the trees.
Using the random intersection trees (RIT) \cite{shah2014random} algorithm, iRF extracts stable interactions of arbitrary order in a computationally efficient way, even when the number of features is large.
There is very positive evidence that iRF extracts predictive, stable, and high-order Boolean interaction information from RF in genomics and other fields \cite{Basu2017,Kumbier2018,cliff2019high}.
While all the works mentioned above provide strong empirical evidence that interactions extracted from the ensemble of decision trees via RF or iRF are informative about underlying biological functional relationships, there are no theoretical results regarding interaction discovery using RF, iRF, or other tree-based methods.
In this paper, as a first step towards understanding the interaction discovery property of tree-based methods, we investigate a key idea in the previous works \cite{wan2009,yoshida2011,Basu2017}, namely, that frequent joint appearance of features on decision paths in the RF tree ensemble suggests an interaction. 

One of the most common assumptions made in previous theoretical analyses of RF is a family of smoothness conditions on the underlying mean regression function, such as the Lipschitz smoothness condition, see e.g., \cite{biau2012, scornet2015, wager2018}. 
However, many biological processes show thresholding or discontinuous interacting behavior among biomolecules \cite{wolpert1969,hoffman2013}, which strongly violates the Lipschitz assumption. It is therefore necessary to introduce a model that can capture the thresholding behavior through discontinuous mean regression function.

\textbf{The \emph{Locally Spiky Sparse (LSS)} model.}
Motivated by this thresholding behavior of biomolecules and inspired by RF's predictive performance successes in genomics data problems \cite{jiang2007mipred,chen2012random,2012Touw}, we consider the locally spiky sparse (LSS) model\footnote{The LSS model was first introduced by authors of \cite{Basu2017} (including one of us) and has already been used in simulations to evaluate the performance of iRF/siRF in \cite{Kumbier2018}.}: 
an additive regression model where the mean regression function is assumed to be a linear combination of Boolean interaction functions.
The linear coefficients, as well as the threshold coefficients of the Boolean functions, are called \textit{model coefficients}.
Via Boolean functions, the LSS model is able to capture discontinuous thresholding behavior in biology, hence it can be more relevant for biologists than models with smoothness constraints.
We believe the LSS model is suitable and useful as a new benchmark model under which to evaluate theoretically (and computationally) interaction discovery performance of tree-based ML algorithms including RF.

\textbf{Our contributions.}
Assume that i.i.d. data samples from the LSS model are given and an RF is fit to this data.

1) For an RF tree ensemble, we first define \textit{signed features}. For a decision path of a set of signed features $S^\pm$ in the ensemble, we then define a new quantity called \textit{depth-weighted prevalence (DWP)}. Intuitively speaking, DWP of $S^\pm$ measures how frequently the features in $S^\pm$ appears together in an RF tree ensemble.
We show that DWP has a universal upper bound that depends only on the size of the set of signed features. 
Moreover, the upper bound is attained with high probability as the sample size increases if and only if the signed features represent a union of interactions in the LSS model. 
Based on DWP, we show that a simple algorithm, i.e., \textit{LSSFind} defined in Algorithm \ref{Algo:1}, can consistently recover interaction components in the LSS model regardless of the model coefficients.

2) Our theoretical results imply that feature subsampling of RF is essential to recover interactions by the RF tree ensemble. 
When too few features are sampled at each node, the tree ensemble is close to extremely randomized trees and DWP of any set of signed features is independent of the response, which means it does not contain information on the LLS model; 
When too many features are sampled, all the trees in the ensemble will be very similar to one another and that turns out to make it difficult to use tree structures to distinguish between interactions and non-interactions.
More specifically, the ratio between the number of subsampled features $m_{try}$ and the total number of features $p$ should be a non-zero constant in order for our algorithm to learn higher-order interactions from tree paths. 

\textbf{Existing theoretical works on RF.}
Existing theoretical studies of RF and its variants belong to two categories. The first focuses on estimating the regression function under Lipschitz or related conditions on the underlying regression function via averaging the decision trees in the RF tree ensemble. The second category studies feature importance measures as an RF output.
In contrast, we provide the first study on feature interaction selection consistency under a new LSS model using DWP extracted from the RF tree ensemble.

In particular, in the first category,
Biau \cite{biau2012} considers ``median forests" \cite{duroux2016}, originally considered as a theoretical surrogate by Breiman \cite{breiman2004}, and obtains the $L_2$ convergence rate under the \textit{Lipschitz} continuous models. 
Scornet et al. \cite{scornet2015} give the first consistency result for Breiman's original RF with \textit{sub-sampling} instead of bootstrapping in the \textit{low-dimensional} setting when data is generated via an additive regression model with continuous components. 
Wager and Athey \cite{wager2018} consider a variant of RF, called honest RF, 
in the causal inference setup and prove its point-wise consistency and asymptotic normality when the conditional mean function is Lipschitz continuous.  
Similar, Mentch and Hooker \cite{mentch2015} showed that,
under some Lipschitz-type conditions, moderately \textit{large number of trees} approximate well the infinite number of trees.
Based on these asymptotic normality results, \cite{mentch2017} derived hypothesis tests for the null hypothesis that the regression function is additive. 
Thus, if one defines \textit{features interaction} as the deviation from a continuous additive regression function, then their results enable testing on a particular candidate.
In contrast, in this work we define \textit{feature interaction} via the non-continuous Boolean functions in the LSS model and we derive consistent interaction selection via the RF tree ensemble, as opposed to a test for an individual interaction as in \cite{mentch2017}.

The second category focuses on theory regarding individual feature importance measures. 
Results in this line of work do not rely on Lipschitz conditions. 
However, to the best of our knowledge, these works study statistical properties of only noisy features, but do not provide results for signal features in finite samples.
Louppe et al.\cite{louppe2013} show that Mean Decrease Impurity (MDI) feature importance for randomized trees has a closed-form formula with infinite number of samples. 
Zhou and Hooker \cite{zhou2020a} use out-of-sample data to improve the MDI feature importance with unbiased theoretical guarantees. 
Li et al. \cite{li2019} show that the MDI feature importance of noisy features is inversely proportional to the minimum leaf node size, and suggest a way to improve the MDI using out-of-bag samples.
L{\"o}cher \cite{loecher2020unbiased} gives a family of MDI feature importance via out-of-bag samples that are unbiased for the noisy features. 
Moreover, many studies focus on permutation-based feature importance measures, in particular, Shapley effects \cite{ishwaran2007,Strobl2008,Janitza2016,nembrini2018,nembrini2019,debeer2020conditional, benard2021}.
Among these works, \cite{benard2021} shows some conceptual similarities to the DWP approach considered in this paper, as they also consider the concept of \textit{joint appearance of features on decision paths} in the RF tree ensemble.
However, instead of using this concept to extract feature interactions, as done in this work, they use it to define an importance sampling scheme to estimate the Shapley effects.

Also related to our work is the recent work \cite{benard2021a}, which analyzes the extraction of rule sets from a RF tree ensemble.
This is very similar to interaction selection as considered in this work, except that the extracted rules in \cite{benard2021a} also include specific estimated thresholds for the individual features.
The theoretical analysis in \cite{benard2021a} focuses on the stability of the selected rules without specifying a particular data generating model.
In contrast, this paper obtains model selection consistency results for LSSFind to estimate signed interactions of signal features under the LSS model.

The rest of the paper is organized as follows: Section \ref{sec:LSS} introduces the LSS model and Boolean interactions in more detail. Section \ref{subsec:RFalgorithm} reviews the RF algorithm and formally defines DWP for a given set of signed features relative to an RF tree ensemble. Section \ref{subsec:mainResults} presents our main theoretical results for DWP and introduces LSSFind, a new theoretically inspired algorithm to detect interactions from RF tree ensembles via DWP. Section \ref{sec:simulations} contains simulation results. We conclude with a discussion in Section \ref{sec:discussion}.

\section{Locally Spiky Sparse (LSS) Model to describe Boolean interactions}\label{sec:LSS}

In this section, we introduce necessary notations and a precise mathematical definition of the LSS model. 
To this end, for an integer $N \in \N$, let $[N] := \{1, 2, \ldots, N\}$. For a set $S$ of finite elements of $[N]$, let $|S|$ denote its cardinality or the number of elements in $S$. For any event $A$, let $\1(A)$ denote the indicator function of $A$. 
 We assume a given data set $\cD = \{(\bx_1, y_1),\ldots, (\bx_n, y_n) \}$ of $n$ samples, with $\bx_i = (x_{i1}, \ldots, x_{in}) \in \R^p$ and $y_i \in \R$. 
We say that the data $\cD$ is generated from a Locally Spiky Sparse (LSS) model when the following assumptions hold true.
\begin{model}
Assume $\cD = \{(\bx_1, y_1),\ldots, (\bx_n, y_n) \}$ are i.i.d. samples from a distribution $P(X,Y)$ such that for some fixed constants $C_\beta > 0, C_\gamma \in (0,0.5)$, the regression function takes the following form:
    \begin{align}\label{eq:E_lssmodel}
        E(Y | X) = \beta_0 + \sum_{j = 1}^J \beta_j \prod_{k \in S_j}\1(X_{k} \gtreqless \gamma_{k})
    \end{align}
    where $\gtreqless$ in \eqref{eq:E_lssmodel} means either $\leq$ or $\geq$, potentially different for every $k$. Coefficients $\beta_j$ are bounded from below, i.e., 
    \begin{align}\label{eq:Cbeta}
    \min_{j=1}^J|\beta_j| > C_\beta     
    \end{align}
     and thresholds $\gamma_j$ are bounded away from 0 and 1, i.e., 
     \begin{align}\label{eq:Cgamma}
         \gamma_j \in (C_\gamma, 1 - C_\gamma),
     \end{align}
     for $j = 1, \ldots, J$. $S_1, \ldots, S_J \subset [p]$ are sets of features called \emph{basic interactions}. We associate $\leq$ in \eqref{eq:E_lssmodel} with a negative sign ($-1$) and $\geq$ with a positive sign ($+1$), such that a \textit{signed feature} can be written as a tuple $(k, b_k) \in [p] \times \{-1, +1\}$. We call $S_1^\pm, \ldots, S_J^\pm\subset [p]\times \{-1,+1\}$ \emph{basic signed interactions} with $S_j^\pm = \{(k, b_k):k\in S_j\}$. 
\end{model}
Note that for interactions with only one feature $k$, due to the sign ambiguity in the LSS model, i.e., $\1(X_k \leq a) = 1 - \1(X_k > a)$, both $\{(k, -1)\}$ and $\{(k, +1)\}$, are counted as an interaction.

The LSS model aims to capture interactive thresholding behavior which has been observed for various biological processes \cite{wolpert1969, ferrelljr1996, little1999, kobiler2005, little2005, levine2008}.
For example, in gene regulatory networks often a few different expression patterns are possible. Switching between those patterns can be associated with individual components that interact via a threshold effect \cite{little1999, kobiler2005, little2005}.
Such a threshold behavior is also observed for other signal transduction mechanisms in cells, e.g, protein kinase \cite{ferrelljr1996} and cell differentiation \cite{wolpert1969}.
Another example of a well studied threshold effect is gene expression regulation via small RNA (sRNA) \cite{levine2008}. 
Although for most biological processes the precise functional mechanisms between different features and a response variable of interest are much more complicated than what the LSS model can capture, 
theoretical investigations of a particular learning algorithm, such as RF, are only feasible within a well defined and relatively simple mathematical model, and useful for practice when such a model is empirically relevant.
Given the empirically observed interactive threshold effects in many real biological systems, the LSS model clearly provides a useful enrichment to the current state of theoretical studies of RF and related methods, since current theoretical models do not capture the often observed interactive threshold behavior.

In order to prove our main Theorem \ref{theo:mainResult}, we further impose the following constraints on the LSS model.
\begin{itemize}
    \item[C1] ({Uniformity}) $X$ is uniformly distributed on $[0,1]^p$.
\end{itemize}
This uniformity assumption implies that each feature is independent of each other. Because any decision tree remains invariant under any strictly monotone transform of an individual feature, the uniform distribution assumption of $X$ can be relaxed to the assumption that individual features $X_j$, $j \in [p]$, are independent with a distribution that has Lebesgue density.
We note that such an independence assumption might be violated in real world problems.
For example, for genetic data with SNPs or gene expression as features $X_j$ there will typically be a strong correlation between features which are located close-by on the chromosome.
However, in many cases, it is feasible to restrict to a subset of features (e.g., those which are located sufficiently far apart on the genome) in order to obtain approximate independence.
In Section \ref{sec:simulations} we also demonstrate in simulations that for sufficiently weak feature correlation one can still obtain accurate interaction selection with LSSFind.
\begin{itemize}
    \item[C2] ({Bounded-response}) $Y$ is bounded, i.e. $|Y| < 1$. 
\end{itemize}
Note that although we assume $|Y| < 1$, the constant $1$ can be changed to any constant as we can scale $Y$ by any positive number and the conclusions in our main results will remain intact. This boundedness condition can be further relaxed to that the residue $Z := Y - E(Y | X)$ is independent of $X$ and $1$-subgaussian if we assume a slightly stronger assumption on $p$ and $n$ than the conditions in C4. 
See Proposition \ref{prop:sub_gaussian} for more detail. 
\begin{itemize}
    \item[C3] ({Non-overlapping basic interactions}) $S_1,\ldots, S_J$ do not overlap, i.e.,
    $ S_{j_1} \cap S_{j_2} = \emptyset \text{ for all } j_1\neq j_2.$
\end{itemize}
The non-overlapping assumption that different interactions $S_{j_1}, S_{j_2}$ with $j_1 \neq j_2$ are disjoint might not always be justified in real world problems. However, it is a crucial assumption for our theorem to hold.
The general problem with overlapping interactions in the LSS model is that such models can be non-identifiable, meaning that different forms of \eqref{eq:E_lssmodel} can imply the same regression function $E(Y | X)$. 
For example, for the response $\1(X_1 < 0.5, X_2<0.5) + \1(X_1 > 0.5, X_2 > 0.5)$, by the definition of signed interactions in the LSS model, it has two basic signed interactions $\{(1, -1), (2, -1)\}$ and $\{(1, +1), (2, +1)\}$. However, we can also write it as $1 - \1(X_1 < 0.5, X_2 > 0.5) - \1(X_1 > 0.5, X_2 < 0.5)$, which has two different basic interactions $\{(1, -1), (2, +1)\}$ and $\{(1, +1), (2, -1)\}$.
This means, a set of signed features which is an interaction in one of the representations is not an interaction in the other.
Due to this identifiability problem, overlapping features can lead to both false positives and false negatives in term of interaction recovery with RF.
One may try to define interaction more broadly to avoid this identifiability problem. For the previous example $\1(X_1 < 0.5, X_2<0.5) + \1(X_1 > 0.5, X_2 > 0.5)$, although the basic signed interactions are not unique, they always constitute of both $X_1$ and $X_2$. Whether the coefficients $\{\beta_j\}_{j=0}^J$ are allowed to have different signs also affects the identifiability. The previous example is identifiable if we only allow positive coefficients. 
For domain problems where interactions are believed to be overlapping, one should investigate different identifiability conditions, but as this depends on the precise application, we leave this for future work. Our work in this paper provides a pathway to investigate this in detail later.
We demonstrate how overlapping features affect our results with a simulation study in Section \ref{sec:simulations}.

In Section \ref{subsec:mainResults} we show that a simple algorithm, LSSFind, that takes an RF tree ensemble as input, can consistently recover basic interactions $S_1,\ldots, S_J$ in the LSS model.
Besides recovering $S_j \subset [p]$, LSSFind can also recover the signs of each feature $k \in \cup_{j = 1}^J S_j $ in the LSS model, which indicates whether the corresponding threshold behavior in \eqref{eq:E_lssmodel} is given by a $\leq$- or a $\geq$-inequality. 
Without loss of generality, in the rest of the paper we assume that all inequalities are $\leq$ in \eqref{eq:E_lssmodel}, that is,
 \begin{align}\label{eq:E_lssmodel_simple}
        E(Y | X) = \beta_0 + \sum_{j = 1}^J \beta_j \prod_{k \in S_j}\1(X_{k} \leq \gamma_{k}).
    \end{align}
We stress, however, that all our results also hold for the general case \eqref{eq:E_lssmodel}.
Because we assume that all the features in basic interactions have minus signs, we denote $S_1^-, \ldots, S_J^- \subset [p] \times \{-1, +1\}$ with $S_j^- = \{(k, -1) \; : \; k \in S_j\}$ as \textit{basic signed interactions} of the LSS model.
As our theoretical results will show, the RF tree ensemble can recover not only the basic interactions $S_j \subset [p]$, but also basic signed interactions $S_j^- \subset [p]\times \{-1, +1\}$.
In other words, through DWP and under the LSS model, the RF tree ensemble can recover not only which features interact with each other in the LSS model, but also whether a particular feature in an interaction has to be larger or smaller than some threshold for this interaction to be active.
Besides basic signed interactions, we also define a \textit{union signed interaction} as a union of individual basic signed interactions, as made more precise in the following definition.
\begin{definition}[Union signed interactions]\label{def:unionInter}
    In the LSS model with basic signed interactions $S_1^-, \ldots, S_J^- \subset [p]\times \{-1, +1\}$, a (non-empty) set of signed features $S^\pm \subset [p]\times \{-1, +1\}$ is called a \textit{union signed interaction}, if 
    \begin{align}
    \begin{aligned}
    {S^\pm} = \bigcup_{j \in \cI} S_j^- \;\bigcup_{j \in \cI_s, k \in S_j, b_k\in \{-1,+1\}}\{(k, b_k)\}
    \end{aligned}
\end{align}
for some (possibly empty) set of indices $\cI \subset \{j \in [J] \;:\; |S_j| > 1 \},\; \cI_s  \subset \{j \in [J] \;:\; |S_j| = 1 \}$.
\end{definition}

In other words,
a union signed interaction is a union of one or more basic signed interactions. For a single-feature signed interaction, its sign-flipped counterpart can also be added to the union.
For example, for an LSS model with $E(Y|X) = \1(X_1 \leq 0.5) + \1(X_2 < 0.5, X_3<0.5),$ there are two basic signed interactions, namely, $\{(1, -1)\}$ and $\{(2, -1), (3, -1) \}$, and five union signed interactions, namely, $\{(1, -1)\}$, $\{(2, -1), (3, -1) \}$, $\{ (1, +1)\}$, $\{(1, -1), (2, -1), (3, -1) \}$, and $\{(1, +1), (2, -1), (3, -1) \} $.

The theoretical results that we present in Section \ref{subsec:mainResults} are asymptotic, in the sense that they assume the sample size $n$ to go to infinity.
Denote the number of signal features $\cup_{j=1}^J S_j$ in the LSS model  to be $s$, i.e., $\sum_{j = 1}^J |S_j| = s$.
We assume $s$ is uniformly bounded regardless of $n$ and $p$.
However, the overall number of features $p$ or the number of noisy features $p - s$ can grow to infinity as $n$ increases.
Our theoretical results also assume 
\begin{itemize}
    \item[C4] (Sparsity) \label{A:asym_n_p} $s = O(1)$ and $\frac{\log(p)}{n} \to 0.$
\end{itemize}
This means that, in contrast to many theoretical works \cite{denil2014, scornet2015, wager2018}\footnote{Note that \cite{biau2012} covers the high dimensional setting, too, but their results only depend on $s$ and not $p$.}, our results hold in a high-dimensional setting as long as the overall number of signal features $s$ is bounded. 
The limit $\frac{\log(p)}{n} \to 0$ is a common assumption for high dimensional settings when analyzing consistency properties of Lasso (see, for instance, \cite{Lasso1996,zhao06a,hastie2015statistical}).

\section{Depth-Weighted Prevalence (DWP) for an RF Tree Ensemble}\label{subsec:RFalgorithm}
In this section, we first review the RF algorithm and then define DWP for a given RF tree ensemble.
\subsection{Review of RF}
RF is an ensemble of classification or regression trees, where each tree $T$ defines a mapping from the feature space to the response. Trees are constructed on a bootstrapped or subsampled data set $\mathcal{D}^{(T)}$ of the original data $\mathcal D$. Note that each tree is conditionally independent of one another given the data. 
Any node $t$ in a tree $T$ represents a hyper-rectangle $R_t$ in the feature space. A split of the node $t$ is a pair $(k_t, \gamma_t)$ which divides the hyper-rectangle $R_t$ into two hyper-rectangles $R_{t,l}(k_t, \gamma_t) = R_t\cap \1(X_{k_t} \le  \gamma_t)$ and $R_{t,r}(k_t, \gamma_t) = R_t\cap \1(X_{k_t} > \gamma_t)$, corresponding to the left child $t_l$ and right child $t_r$ of node $t$, respectively. For a node $t$ in a tree $T$, $N_n(t) = |\{i\in \mathcal{D}^{(T)}:\bx_i \in R_t\}|$ denotes the number of samples falling into
$R_t$.

Each tree $T$ is grown using a recursive procedure (denoted as CART algorithm \cite{Breiman2001}), which proceeds in two steps for each node $t$. First, a subset $\Mt \subset [p]$ of features is chosen uniformly at random. The size of $\Mt$ is $\mt$. Then the optimal split $k_t \in \Mt, \gamma_t \in \R$ is determined by maximizing impurity decrease defined in \eqref{eq:impdecrease}:
\begin{equation}
\label{eq:impdecrease}
    \Delta_I^n(t) := I_n(t) - \frac{N_n(t_l)}{N_n(t)}I_n(t_l) - \frac{N_n(t_r)}{N_n(t)}I_n(t_r)
\end{equation}
where $t_l$($t_r$) is the left(right) child of $t$ and for sample size $n$, $I_n(t)$ is the impurity measure defined in this paper as 

$$ I_n (t) = \mbox{variance of } \{y_i, i \in R_t\},$$
which is the variance of the response $y_i$'s for all the samples in the region $R_t$. Note that the analysis of this paper holds only for the variance impurity measure but it is possible to extend to other impurities measures, which is left as future work.  The procedure terminates at a node $t$ if two children contain too few samples, e.g., $\min\{N_n(t_l), N_n(t_r)\} \leq 1$, or if all responses are identical, e.g., $I_n(t) = 0$. For any tree $T$ and any leaf node $t_{\leaf}\in T$, denote $\cp(t_\leaf)$ to be a path to that leaf node.

\begin{definition}[Depth of a path]
Given a path $\cp(t_\leaf)$ that connects root node $t_{1}$ and leaf node $t_{\leaf}$ in a tree $T$, we define the depth of the path $\cp(t_\leaf)$ to be the number of non-root nodes contained in the path.
\end{definition}

For any hyper-rectangle $R_t$, $\mu(R_t)$ denotes its volume.
We make the following assumptions on an RF tree ensemble:
\begin{assumption}[increasing depth of a tree in the RF ensemble]\label{A:increasing_depth}
The minimum depth of any path in any tree goes to infinity, i.e., \[\min_{T}\min_{t_\leaf\in T} D(t_\leaf)\pto \infty\] as $n\to \infty$.
\end{assumption}
\begin{assumption}[balanced split in a tree of the RF ensemble]\label{A:balancedsplit}
Each split $(k_t, \gamma_t)$ is balanced: for any node $t$,  $$\min\left(\frac{\mu(R_{t,l}(k_t, \gamma_t))}{\mu(R_{t,r}(k_t, \gamma_t))},\frac{\mu(R_{t,r}(k_t, \gamma_t))}{\mu(R_{t,l}(k_t, \gamma_t))}\right) > \frac{C_\gamma}{1-C_\gamma}.$$
\end{assumption}
Note that, without loss of generality, we use the same $C_\gamma$ here as in the LSS model. Otherwise, we can always let $C_\gamma$ to be the minimum of the two.
\begin{assumption}[$\mt$ is of order $p$]\label{A:mtry}
$C_m p + (1 - C_m)s \leq \mt \leq (1 - C_m)(p -  s )$ where $C_m \in (0,0.5)$ is a constant.
\end{assumption}
\begin{assumption}[no bootstrap or subsampling of samples]\label{A:no_bootstrap}
All the trees in RF are grown on the whole data set without bootstrapping or subsampling, i.e. $\cD^{(T)} = \cD$ for any $T$.
\end{assumption}

\ref{A:no_bootstrap} is a technical assumption that simplifies our notation and analysis. 
We assume that each tree is grown using all of the samples, which is quite different from the assumptions on subsampling in recent theoretical works on RF (e.g. \cite{biau2012} and \cite{wager2018}). 
The subsampling rate plays a crucial role in the analysis of the asymptotic distribution of the RF predictor \cite{biau2012, wager2018}, where it is assumed that the subsampling rate converges to zero at a desirable rate.
However, since we focus on the features selected at each node and not on the asymptotic distribution of the predictor, we do not require such assumptions on the subsampling rate.

\ref{A:increasing_depth} ensures that the length of any decision path in any tree tends to infinity. This assumption is reasonable as tree depths in RF is usually of order $O(\log n)$ which tends to infinity as $n\to \infty$. \ref{A:balancedsplit} ensures that each node split is balanced. Similar conditions are used commonly in other papers \cite{wager2018}. \ref{A:mtry} shows the important role of the parameter $\mt$. Roughly speaking, $\mt$ cannot be too small or too big. When $\mt$ is too small, there will be too many splits on irrelevant features which makes the tree noisy. When $\mt$ is too big, there will be too little variability in the tree ensemble. This motivation will be made rigorous in the proof of Theorem \ref{theo:mainResult}. 
\subsection{Depth weighted prevalence (DWP)}\label{subsec:InteractionPrevalence}

In this section, for a tree ensemble from RF, we formally introduce Depth Weighted Prevalence (DWP). 
Given a decision tree $T$ in an RF tree ensemble, we can randomly select a path $\cP$ of $T$ as follows: we start at the root node of $T$ and then, at every node, randomly go left or right until we reach a leaf node. 
This is equivalent to selecting a path in $T$ of depth $D$ with probability $2^{-D}$ from all the paths in a decision tree.
Denote the nodes in $\cP$ to be $t_1,\ldots, t_D, t_\leaf$. 
As such, any path $\cP$ in a decision tree $T$ can be associated with a sequence of signed features $(k_{t_1}, b_{t_1}), \ldots, (k_{t_D}, b_{t_D}) \in [p] \times \{-1,+1\}$, where $D$ is the depth of the path and for any inner node $t \in [D]$ on the path the sign $b_{t}$ indicates whether the path at node $t$ followed the $\leq$ direction ($b_{t} = -1$) or the $>$ direction ($b_{t} = +1$) for the split on feature $k_t \in [p]$.
For a given RF tree ensemble depending on data $\cD$, the randomly selected path $\cP$ of tree $T$ and any fixed constant $\epsilon > 0$, we now define $\hat{\cF}_{\epsilon}(\cP, T,\cD)$ to be the set of signed features on $\cP$ where the corresponding node in the RF had an impurity decrease of at least $\epsilon$, that is,
\begin{align}
\begin{aligned}
    &\hat{\cF}_{\epsilon}(\cP, T,\cD) := \{ (k_t, b_{t}) \; | \;  \text{ $t$ is an inner node of $\cP$} \\ 
    & \text{ with }\Delta_I^n(t) > \epsilon
   \text{ and feature } k_t \text{ appears first time on $\cP$} \}.
   \end{aligned}\label{eq:define_hat_cf}
\end{align}
We use $\hat{\cF}_{\epsilon}$ as a shorthand for $\hat{\cF}_{\epsilon}(\cP, T, \cD)$ when the path $\cP$ from tree $T$ and the data $\cD$ of interest are clear. Note that if a feature appears more than once on the path $\cP$, its sign in $\hat \cF_\epsilon $ is the sign when the feature appears the first time with the impurity decrease above the threshold.
Our main theorem will be stated in terms of the DWP of a signed feature set $S^\pm \subset [p]\times \{-1, +1\}$ on the random path $\cP$ within $\hat{\cF}_{\epsilon}$. To formally define the DWP of $S^\pm$, we first need to identify the sources of randomness underlying $\hat{\cF}_{\epsilon}$. There are three layers of randomness involved:
\begin{enumerate}
    \item \textbf{($\cD$: Data randomness)} the randomness involved in the data generation;
    \item \textbf{($T$: Tree randomness)} the randomness involved in growing an individual tree with parameter $\mt$, given data $\cD$;
    \item \textbf{($\cP$: Path randomness)} the randomness involved in selecting a random path $\cP$ of depth $d$ with probability $2^{-d}$, given a tree $T$ from an RF tree ensemble with parameter $\mt$ based on data $\cD$. 
\end{enumerate}
In the following definition of the DWP of signed feature sets, the probability is conditioned on data $\cD$, and taken only over the randomness of the tree $T$ and the randomness of selecting one of its paths as in $\cP$.

\begin{definition} (Depth-Weighted Prevalence (DWP))
Conditioning on data, for any signed feature set $S^\pm \subset [p]\times \{-1, +1\}$, we define the \emph{Depth-Weighted Prevalence (DWP)} of $S^\pm$ as the probability that $S^\pm$ appears on the random path $\cP$ within the set $\hat{\cF}_\epsilon, $ that is, 
\begin{align} \mathrm{DWP}_{\epsilon}(S^\pm) =& P_{(\cP, T) }(S^\pm \subset \hat\cF_{\epsilon} \;|\; \cD).\label{eq:def_dwp}
\end{align}
We emphasize that the probability of selecting a path in a tree $T$ is $P(\cP|T) = 2^{-d}$ where $d$ is the depth of the path $\cP$.
\end{definition}

While we only have a fixed sample size which means the data randomness is inevitable, the tree randomness and path randomness are generated by the algorithm and thus can be eliminated by sampling as many trees and paths as we like. 
Because the DWP in \eqref{eq:def_dwp} is only conditioned on data, for any given $\epsilon > 0$ and set of signed features $S^\pm$, it can be computed with arbitrary precision from an RF tree ensemble with sufficiently many trees (recall that, conditioned on data $\cD$, the different trees in an RF tree ensemble are generated independently). 

\section{Main results}\label{subsec:mainResults}
In this section we present our main theoretical results which are concerned with DWP as introduced in the previous section.
Our results show that LSSFind (Algorithm \ref{Algo:1}), which is based on DWP at an appropriate level $\epsilon$ described in Theorem \ref{Theo:mainResult_final}, consistently recovers signed interactions under an LSS model.
Before we state our main results in full detail, we want to illustrate it with a simple example.

\begin{algorithm2e}
 \caption{LSSFind($\mt$,$\epsilon$, $\eta$, $s_{\max}$)}\label{Algo:1}
\SetKwInOut{Input}{Input}
\Input{Dataset $\cD$, RF hyperparameter $m_{try}$, impurity threshold $\epsilon > 0$, prevalence threshold $\eta > 0$, and maximum interaction size $s_{\max} \in \N$.}
\SetKwInOut{Output}{Output}
\Output{A collection of sets of signed features.}
 Train an RF using dataset $\cD$ with parameter $m_{try}$.\;
 
 return $\{S^\pm \subset [p]\times \{-1, +1\} \text{ such that } |S^\pm| \leq s_{\max} \text{ and } 2^{|S^\pm|} \cdot \mathrm{DWP}_{\epsilon}(S^\pm) \geq 1 - \eta \}.$
\end{algorithm2e}

\textbf{Illustrative example:}
Assume that $p = 2$ and there are just two features $X_1$ and $X_2$. Assume there is a single interaction $J = 1$ 
and the regression function is given by
\begin{align}\label{eq:response_simple_example}
    E(Y | X_1, X_2) = \1(X_1 \leq 0.5) \cdot \1(X_2 \leq 0.5).
\end{align}
The response surface of \eqref{eq:response_simple_example} is shown in Figure \ref{fig:trees_simple_example} in the top middle plot.
We consider the population case, where we have full access to the joint distribution $P(X,Y)$, that is, we have access to an unlimited amount of data ($n = \infty$).
When we apply the RF algorithm as in Section \ref{subsec:RFalgorithm}, then for each individual tree in the forest the root node either splits on feature $X_1$ or on feature $X_2$. 
Since $X_1$ and $X_2$ are completely symmetric in the distribution $P(X,Y)$, thus, if the RF algorithm grows more and more trees, in the limit, half of them will split on $X_1$  at the root node and half of them split on $X_2$ at the root node. For infinite data, this 50/50 split is introduced by the CART algorithm since the two splits have identical decreases of impurities.
Furthermore, the split at any node will be at $0.5$ for any of the two features, since the two splits corresponding to $X_1 \leq 0.5$ and $X_2 \leq 0.5$ maximize the impurity decrease given infinite data.
This is illustrated in Figure \ref{fig:trees_simple_example}, where the left bottom figure shows a tree which splits on feature $X_1$ at the root node and the right bottom figure shows a tree which splits on feature $X_2$ at the root node.
As each tree in RF grows to purity, when the root node splits at feature $X_1$, then for the path of the tree which follows the $(1, +1)$ direction, that is, the $X_1 > 0.5$ direction, the tree will stop growing, as the respective response surface is already constant.
However, for the path of the tree which follows the $(1, -1)$ direction, that is, the $X_1 \leq 0.5$ direction, the tree will further split on the remaining feature $X_2$. Then the tree will stop because the node reaches purity.
Thus, we conclude that the forest consists of exactly the two different trees shown in Figure \ref{fig:trees_simple_example} and in the limit, where the number of trees grows to infinity, each of the two trees appears equally often.
\begin{figure}[t]
    \centering
    \includegraphics[width = 0.45\textwidth]{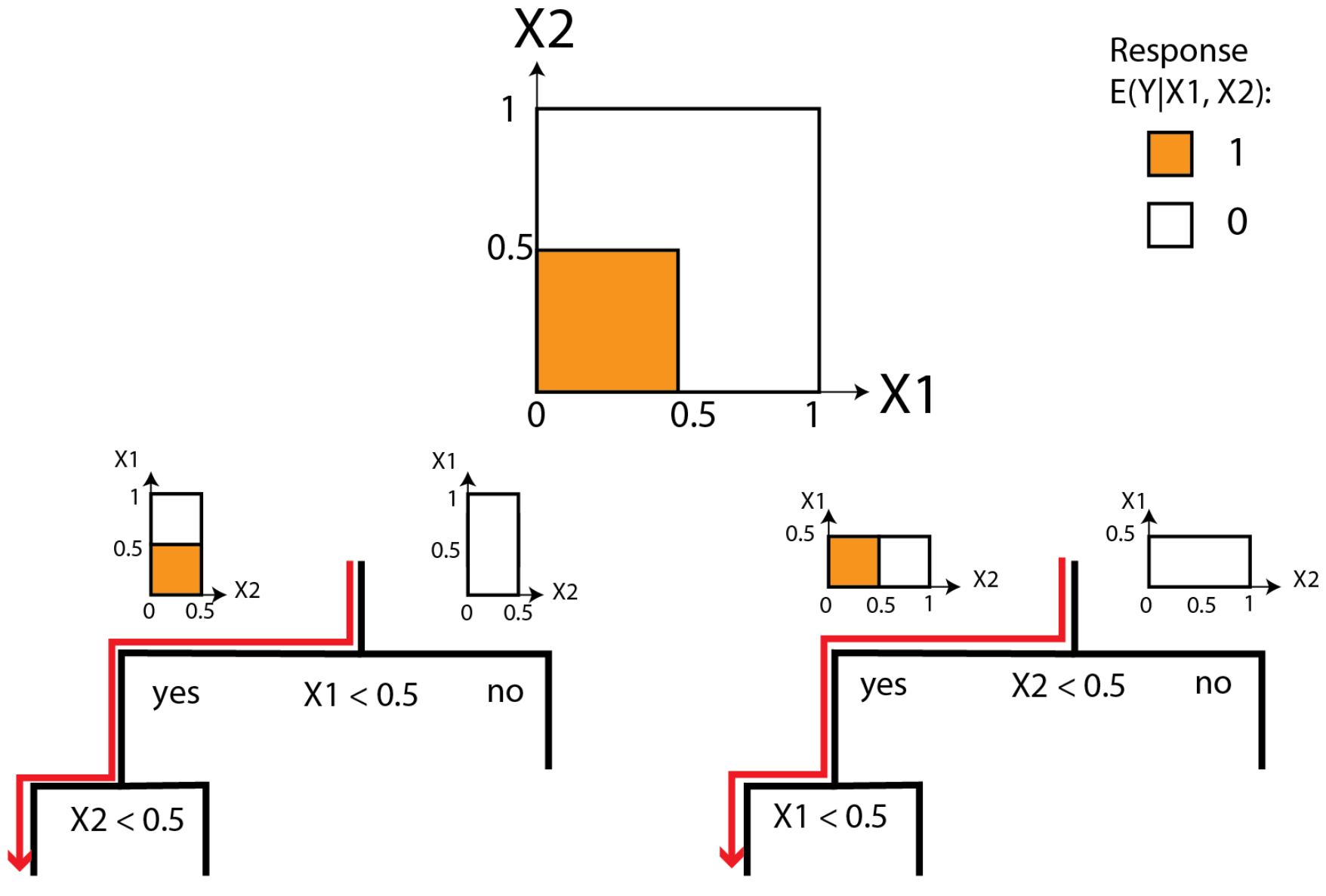}
    \caption{
    Exemplary RF decision trees trained on data as in \eqref{eq:response_simple_example} to illustrate the results that will appear in Theorem \ref{theo:mainResult}.
    Top center: response surface of $E\left( Y \middle| \; X_1, X_2 \right)$ as in \eqref{theo:mainResult} with $X_1 \in [0,1]$ on the x-axis and $X_2 \in [0,1]$ on the y-axis. 
    Bottom left: a decision tree that splits on feature $X_1$ at the root node with the respective regions and conditional response surfaces for left and right child of the root node. 
    Bottom right: a decision tree that splits on feature $X_2$ at the root node. The red-marked decision paths contain all signed features from the basic signed interaction $S^- = \{(1,-), (2, -)\}$ from an LSS model as in \eqref{eq:response_simple_example}.
    For both of the trees, if one starts at the root node and randomly goes left or right at every node, then the probability of the basic signed interaction to appear on the path is $\mathrm{DWP}_{\epsilon}(S^-) = 2^{-2} = 2^{-|S^-|}$.
    In contrast, for any other set of signed features $S^\pm \subset [p] \times \{-1,+1\}$ it holds that $\mathrm{DWP}_{\epsilon}(S^\pm) < 2^{-|S^\pm|}$.
    This provides a simple example for the more general result in Theorem \ref{theo:mainResult}. }
    \label{fig:trees_simple_example}
\end{figure}

For each node $t$ in these trees, the impurity decrease satisfies $\Delta_I^n(t) \geq 1/16$.
Thus, for any $\epsilon < 1/16$, we can show that the DWP of the basic signed interaction $S^- = \{(1, -1), (2, -1)\}$ is $2^{-|S^-|}$. To show this, we can get:
\begin{align*}
&\mathrm{DWP}_{\epsilon}(S^-) = P(S^-\subset \hat{\cF}_\epsilon | \cD)\\
&=\underbrace{P_T(\text{$T$'s root splits on feature $1$})}_{=0.5, \text{correspond to the left tree}}\\
&\cdot \underbrace{P(S^-\subset \hat{\cF}_\epsilon|\cD, \text{$T$'s root splits on feature $1$})}_{=0.25, \text{ only the red path satisfies this.}} +\\
&\underbrace{P_T(\text{$T$'s root splits on feature $2$})}_{ = 0.5, \text{correspond to the right tree}}\\ &\cdot  \underbrace{P(S^-\subset \hat{\cF}_\epsilon | \cD,\text{$T$'s root splits on feature $2$})}_{=0.25, \text{ only the red path satisfies this.}}\\
&=0.5 \cdot 2^{-2} + 0.5 \cdot  2^{-2} =
2^{-2} = 2^{-|S^-|}.
\end{align*}

In the above example with infinite data, the tree depth is not going to infinity, which means it does not satisfy assumption \ref{A:increasing_depth}. \ref{A:increasing_depth} is needed only for the finite sample case because, for finite samples, internal nodes in a tree can never reach purity due to noise.

In Figure \ref{fig:trees_simple_example} the paths which contain the basic signed interaction $ S^- = \{(1, -1), (2, -1)\}$ are marked red. 
 For all the other sets of signed features $S^\pm \subset [p]\times \{-1, +1\}$, it is easy to check that
 $\mathrm{DWP}_{\epsilon}(S^\pm) < 2^{-|S^\pm|}.$
 For example, 
  \begin{align*}
     \mathrm{DWP}_{\epsilon}(\{ (1, -1), (2, +1) \}) = 0.5 \cdot  2^{-2} + 0.5 \cdot 0 < 2^{-2}
 \end{align*}
 and 
   \begin{align*}
     &\mathrm{DWP}_{\epsilon}(\{ (1, -1) \}) \\
     &=\underbrace{P_T(\text{$T$'s root splits on feature $1$})}_{=0.5, \text{correspond to the left tree}}\\
     &\cdot \underbrace{P(\{ (1, -1) \}\subset \hat{\cF}_\epsilon|\cD, \text{$T$'s root splits on feature $1$})}_{=0.5, \text{ any path that goes left at the root satisfies this.}} +\\
&\underbrace{P_T(\text{$T$'s root splits on feature $2$})}_{ = 0.5, \text{correspond to the right tree}}\\ &\cdot  \underbrace{P(S^-\subset \hat{\cF}_\epsilon | \cD,\text{$T$'s root splits on feature $2$})}_{=0.25, \text{ only the red path satisfies this.}}\\
&= 0.5 \cdot 2^{-1} + 0.5 \cdot 2^{-2} < 2^{-1}.
 \end{align*}

  As we will formally state in the two theorems below, the same reasoning holds true asymptotically for any RF trained on the data from the LSS model, namely, the DWP of a set of signed features $S^\pm \subset [p]\times \{-1, +1\}$ is always upper bounded by $2^{-|S^\pm|}$ and this upper bound is attained if and only if $S^\pm$ is a union signed interaction.
Recall that the DWP depends on the data $\cD$.
It turns out, that the general upper bound follows directly from the construction of DWP and holds for any data $\cD$, i.e., independent of the LSS model, as the following theorem shows.
\begin{theorem}\label{theo:mainResult_general_upper_bound}
For any impurity threshold $\epsilon > 0$ and any set of signed features $S^\pm \subset [p]\times \{-1, +1\}$ for the RF algorithm from Section \ref{subsec:RFalgorithm} it holds true that
\begin{itemize}
 \item (General upper bound) \quad  $\mathrm{DWP}_{\epsilon}(S^\pm) \leq 2^{- |S^\pm|}.$
\end{itemize}
\end{theorem}
In addition, when the data $\cD$ is generated from an LSS model, asymptotically (as the sample size increases) the general upper bound is attained if and only if $S^\pm$ is a union signed interaction, as the following theorem shows.
\begin{theorem}\label{theo:mainResult}

Assume that the data $\cD$ is generated from an LSS model with uniformity, bounded-response, non-overlap basic interactions, and sparsity constraints (see C1 - C4).
For any impurity threshold $\epsilon > 0$, let
\begin{align}\label{eq:epsilon_tilde}
    b(\epsilon) := \left({4 \epsilon}/({ C_\beta^2 C_\gamma^{2s - 1}  })\right)^{ C_m^{2 s} / \log(1/C_\gamma) },
\end{align}
with constants $C_\beta$ as in \eqref{eq:Cbeta}, $C_\gamma$ as in \eqref{eq:Cgamma}, $s$ as in C4, and $C_m$ as in \ref{A:mtry}.
Given a set of signed features $S^\pm \subset [p]\times \{-1, +1\}$, for the RF algorithm from Section \ref{subsec:RFalgorithm} it holds true that,
\begin{itemize}
    \item (Interaction lower bound) when $S^\pm$ is a union signed interaction as in Definition \ref{def:unionInter}, we have
    \[ \mathrm{DWP}_{\epsilon}(S^\pm) \geq  2^{- |S^\pm|} - b(\epsilon) - r_n(\cD, \epsilon); \]
    \item (Non-interaction upper bound) when $S^\pm$ is not a union signed interaction, then, 
      \[ \mathrm{DWP}_{\epsilon}(S^\pm) \leq 2^{-|S^\pm|}\left( 1  -   \frac{ C_m^{s}}{2}\right)  + r_n(\cD, \epsilon), \]
      with
      \[r_n(\cD, \epsilon) \pto 0 \quad \text{ as } n \to \infty,\]
      where $\pto$ denotes convergence in probability.
\end{itemize}
\end{theorem}

\textbf{Proof Sketch}:  
The detailed proof of Theorem \ref{theo:mainResult} is deferred to Section \ref{theo:mainResult}. It has two major parts: first, showing the assertion for the idealized population case and second, extending the population case to the finite sample case. 

In the first part, we define a population version of the set $\hat{\cF}_\epsilon$, which we denote as $\cF$.
 The set $\cF$ only contains \textit{desirable features}, which are features of a path $\cP$ that correspond to a positive decrease in impurity if the RF gets to see the full distribution $P(X, Y)$ (not just a finite sample $\cD$). Note that desirable/non-desirable features are different from signal/noisy features. The definition of desirable/non-desirable features depends on the concerned path in a tree. A noisy feature is always a non-desirable feature but a signal feature can become a non-desirable feature when it has been split in the path. See Definition \ref{Def:desirable_features}.
The set $\cF$ is an oracle, in the sense that its construction depends on the true underlying LSS model.
This is in contrast to the set $\hat{\cF}_\epsilon$, which can be computed for any given path from a tree of RF.
Given this definition of $\cF$,
a sketch of the proof of the major assertions of Theorem \ref{theo:mainResult_general_upper_bound}
 and \ref{theo:mainResult} is as follows: 
\begin{enumerate}
    \item When a set of signed features $S^\pm$ appears in $\cF$, this implies that every time a signed feature $(k, b) \in S^\pm$ appears on the way from the root node to the leaf, the splitting direction implied by $b$ was selected for $\cP$, which gives rise to the general upper bound of $\mathrm{DWP}_{\epsilon}(S^\pm) \leq 2^{-|S^\pm|}$ (see Theorem \ref{theo:mainResult_general_upper_bound}).
    \item If $S^\pm$ is a union interaction, then (assuming all leaf nodes of the tree are pure) a correct splitting direction for each of its features already implies that $S^\pm$ appears on $\cP$ and thus, $\mathrm{DWP}_{\epsilon}(S^\pm) \approx 2^{-|S^\pm|}$ (see first part of Theorem \ref{Theo:mainResult_final} below).
    \item If $S^\pm$ is not a union interaction, then there will always be the possibility that, although every split for an encountered feature which is an element of $S^\pm$ was done in the correct direction, some of the features in $S^\pm$ were just never encountered and therefore, a correct splitting direction does not imply that $S^\pm$ appears on $\cP$, hence $\mathrm{DWP}_{\epsilon}(S^\pm) < 2^{-|S^\pm|}$ (see second part of Theorem \ref{Theo:mainResult_final}).
\end{enumerate}

In the second part of the proof,  we show that the observed set $\hat{\cF}_\epsilon$ and the oracle set $\cF$ are the same with probability going to 1 as $\epsilon$ 
goes to zero
and $n$ goes to infinity. That would be nice and easy if a tree grown using finite samples will converge to a tree grown using the population in terms of the splitting features and thresholds when sample size tends to infinity. However, that is not true. The obstacle is that, when a node splits on a non-desirable feature, since all the thresholds yield the same impurity decrease in the population case, the threshold selected via finite samples can deviate from the threshold via the population no matter how many samples are used. Thus, we need to carefully analyze desirable features and non-desirable features separately based on uniform convergence results. \qed 

\textbf{Remark 1}: Theorems \ref{theo:mainResult_general_upper_bound} and \ref{theo:mainResult} demonstrate that recovery of interactions becomes exponentially more difficult as the size of an interaction increases. An interaction $S^\pm$ corresponds to a region of size $O(2^{-|S^\pm|})$, which means the sample size must be much larger than $2^{|S^\pm|}$ to have enough samples in that region. Also, the DWP at an appropriate level $\epsilon$ of a basic interaction $S^\pm$ is $2^{-|S^\pm|}$. To have a consistent estimate, the number of independent paths should be much larger than $2^{|S^\pm|}$. Thus, when one wants to recover an interaction of size $s$, the number of samples and the number of trees must be much larger than $2^s$. That shows the intrinsic difficulty of estimating high order interactions.

Using the conclusions in Theorem \ref{theo:mainResult}, one can show that LSSFind (Algorithm \ref{Algo:1}) can consistently recover all the basic interactions from the LSS model, as stated in Theorem \ref{Theo:mainResult_final}.
\begin{theorem}\label{Theo:mainResult_final}
Let the output of LSSFind (Algorithm \ref{Algo:1}) be $\mathscr S (\mt, \epsilon, \eta, s_{\max})$. Under the same settings as in Theorem \ref{theo:mainResult}, as long as $\mt, \epsilon, \eta$ satisfies the assumptions in Theorem \ref{theo:mainResult} and the following condition:
\begin{align}\label{eq:codEtaEps}
  2^s \cdot b(\epsilon) <\eta < \frac{[C_m]^s}{2},
\end{align}
with $b(\epsilon)$ defined in \eqref{eq:epsilon_tilde} and $C_m$ in Assumption \ref{A:mtry}, then, with probability approaching 1 as $n\to \infty$, $\mathscr S$ is a superset of the basic signed interactions with size at most $s_{\max}$ and a subset of union signed interactions.
In particular, if we define 
\begin{align*}
    \mathscr U = \{S\in \mathscr S\mid \text{There is no set $S'\in \mathscr S$ s.t. }, S\subsetneq S'\},
\end{align*}
then $\mathscr U$ equals the set of basic signed interactions of size at most $s_{\max}$.
\end{theorem}
Note that to recover all the basic interactions, $s_{\max}$ needs to be larger than or equal to the order of all the basic signed interactions but the latter is unknown as we do not know the underlying LSS model.

\textbf{Proof:}
If $S^\pm$ is not a union signed interaction,
then it follows from the second part of Theorem \ref{theo:mainResult} and $\eta < [C_m]^s/2$ that
$2^{|S^\pm|} \cdot \mathrm{DWP}_{\epsilon}(S^\pm)  < 1 -\eta, $
with probability approaching $1$ as $n\to\infty$.
Thus, $\mathscr S$ is a subset of union signed interactions.
If $S^\pm$ is a basic signed interaction of size at most $s_{\max}$, then it follows from the first part of Theorem \ref{theo:mainResult} and the fact that $ 2^s \cdot b(\epsilon) < \eta $ that 
$2^{|S^\pm|} \cdot \mathrm{DWP}_{\epsilon}(S^\pm)  \geq 1  -\eta,$
with probability approaching $1$ as $n\to\infty$.
Thus, $\mathscr S$ is a superset of the basic signed interactions with size at most $s_{\max}$. \qed 

\textbf{Remark 2}: 
One important assumption in our theorem is the sparsity of signal features. 
If there are many ``weak" signal features, it is very hard for RF to work well. For RF, at each node of a tree, only one feature is used. That means the total number of features used along each path is limited by the depth of the tree, which is usually of order $O(\log n)$.  
For our assertions of Theorem \ref{theo:mainResult} the hard threshold $\epsilon$ in the set $\hat{\cF}_\epsilon$ has the purpose to select the signal features.
Clearly, the choice of an appropriate value of $\epsilon$ is hard in practice.
The iterative Random Forest fitting procedure in iRF \cite{Basu2017} (which uses joint prevalence on decision paths in RF to recover interactions, similar as suggested by Theorem \ref{theo:mainResult}) filters noisy features not with a hard, but with a soft thresholding procedure: it grows several RF iteratively and samples features at each node according to their feature importance from the previous iteration.
In that way, one does not need to chose a single hard threshold, which leads to a much more practical algorithm. 
Unfortunately, such an iterative soft thresholding makes theoretical analysis much harder, which is why we restrict to the hard threshold for the theoretical analysis in this work.

One of the remarkable aspects of the result in Theorem \ref{Theo:mainResult_final} is that the range of $\eta$ is independent of any model coefficients in the LSS model (that is, the linear $\beta$ coefficients and the $\gamma$ thresholds). 
For sufficiently small $\epsilon$, it only depends on the number of signal features $s$ and the bound of $\mt$, i.e., $C_m$, and nothing else.
In a sense, this shows that the tree ensemble of RF contains the \textit{qualitative} or discrete-set information of \textit{which features interact with each other}, independently of the \textit{quantitative} information about \textit{what are the numerical parameters or model coefficients} in the LSS model.

Another interesting aspect about the results from Theorem \ref{Theo:mainResult_final} is that it sheds some light on the influence of $m_{try}$ on the interaction recovery performance of RF.
For the third assertion in Theorem \ref{theo:mainResult} we actually show that
\begin{align*}
    &\mathrm{DWP}_{\epsilon}(S^\pm) \leq r_n(\cD, \epsilon)\; +\\ 
    &0.5^{|S^\pm|}\left( 1  -   0.5 \min_{k \in \cup_j S_j}P(\text{root node splits on feature }k)\right).
\end{align*}
When $m_{try}$ is too large, $\min_{k \in \cup_j S_j}P( \text{root node splits on feature }k)$ can get very small, as particularly strong features (large initial impurity decrease) can mask weaker features.
As an extreme example, consider the situation where $m_{try} = p$ and thus, the root node gets to see all the features. In that case, the single feature which has the highest impurity decrease, say $X_1$, will \textit{always} appear at the root node and hence, for $S^\pm = \{(1, -1)\}$ or $S^\pm = \{(1, +1)\}$ one will get $\mathrm{DWP}_{\epsilon}(S^\pm) = 2^{-|S^\pm|} = 0.5$, independent of whether $S^\pm$ is an interaction or not.
This shows that when $m_{try}$ is too large, DWPs corresponding to false interactions can attain the universal upper bound $ 2^{-|S^\pm|} $, which leads to false positives in terms of interaction recovery.
On the other hand, when $m_{try}$ is too small, for a signal feature $k \in \cup_j S_j$ it can take a long time until it gets selected into the candidate feature set at a node.
In particular, for finite sample, it can happen that the tree reaches purity due to lack of samples without having split on any of the signal features.
Hence, the reasoning of Theorem \ref{theo:mainResult}, namely that \textit{correct split direction} $+$ \textit{pure path} implies that a union interaction appears on the path does not hold anymore. 
This can lead to union interactions having significantly smaller DWP than the universal upper bound $ 2^{-|S^\pm|} $, i.e., false negatives in terms of interaction recovery.

\section{LSSFind and simulation results}\label{sec:simulations}
In this section, motivated by our theoretical results in the previous section, we evaluate LSSFind empirically in terms of its ability to recover interactions\footnote{Source code is available at https://github.com/Yu-Group/interaction\_selection}.
Simulated experiments are carried out to assess the ability of LSSFind to correctly recover interactions from the LSS model, even when some of the LSS model assumptions are violated.

In LSSFind, one needs to search over all possible sets with size at most $s_{\mathrm{max}}$ to obtain the final result. That is computationally very intensive. 
One more efficient way is to only look for sets with size at most $s_{\mathrm{max}}$ and also with \begin{align} \label{Eq:filter_condition}
&\mathrm{DWP}_\epsilon(S^\pm) \geq (1-\eta)\cdot 2^{-s_{\mathrm{max}}}.
\end{align}
which implies that we don't need to search over all possible sets with 
sizes at most $s_{\mathrm{max}}$; instead, we need to search only for sets whose $DWP_{\epsilon}$'s are larger than $(1-\eta)\cdot 2^{-s_{\mathrm{max}}}$. Because many sets with 
sizes at most $s_{\mathrm{max}}$ are filtered out, this significantly reduces the search space.
We use the FP-growth algorithm \cite{han2000mining} to obtain those sets of signed features which have a DWP higher than some threshold. Note that DWP requires infinite number of trees. To approximate DWP, we use 100 trees in the simulation. Since each tree contains thousands of paths, we have hundreds of thousands of paths to estimate the DWP for.

\subsection{Simulated data from LSS models}\label{subsec:simSection}
{
In the following we present simulation results, where we generated data $\cD$ from the LSS model for different number and order of basic interactions and different signal-to-noise (SNR) ratios.
We find that LSSFind recovers the true interactions from the LSS model with high probability, whenever the overall number of basic interactions and their orders are small.
}

More precisely, we consider $p = 20$ features and $n = 1,000$ sample, where each feature $X_j$ is generated from an uniform distribution $U([0,1])$, independent from one another. 
The number of basic interactions is denoted as $J$ and the order of each interaction is denoted by $L$. We consider the same threshold $\tau$ for all features. The noise is Gaussian with variance $\sigma^2$ and the response is:
\begin{align}\label{eq:simulations_model_lss}
    Y = \sum_{j=1}^J \prod_{k=(j-1)\cdot L + 1}^{k\cdot L}\1({X_{k} < \tau}) + \mathcal{N}(0, \sigma^2).
\end{align}
We consider different values for $J$, $L$, and $\sigma^2$, namely, $J=1,2$, $L=2,3, 4$, and $\sigma^2$s such that the signal-to-noise ratios (SNR) is $0.5, 1, 2$, or $5$. For a given $J$ and $L$, the threshold $\tau$ is chosen such that about 50 percent of samples fall into the union of hyper-rectangles, that is, $\cup_{j=1}^J\cap_{k=(j-1)\cdot L + 1}^{j\cdot L}\{X_{k} < \tau\}$. As we know that the number of samples falling into $\cup_{j=1}^J\cap_{k=(j-1)\cdot L + 1}^{j\cdot L}\{X_{k} < \tau\}$, which can also be roughly thought as the label imbalance, has a high impact on the results, keeping this number the same across different simulation settings makes sure that the simulation outcome are more comparable. The results are averaged across 40 independent Monte Carlo runs.
We grow RF using the scikit-learn package with $100$ trees. We apply LSSFind with parameters $\eta = 0.01$, $\epsilon = 0.01$, and $s_{\max} = L+1$. Recall that we use \eqref{Eq:filter_condition} to select candidate interactions. If $s_{\max}$ is set to $L$, the condition \eqref{Eq:filter_condition} would be too restrictive for challenging situations, such as when LSS model is violated, and 
LSSFind can end up finding no interactions.
Given a set $\mathscr S^*$ of $K$ true basic signed interactions from the respective LSS model and output from LSSFind $\mathscr S$, we evaluate their proximity based on their Jaccard distance:
\begin{align}\label{eq:proxScore}
\mathrm{score}(\mathscr S^*, \mathscr S) = \frac{|\mathscr S^*\cap \mathscr S|}{|\mathscr S^*\cup \mathscr S|}.
\end{align}
Note that any element in $\mathscr S^*$ and $\mathscr S$ is a set of signed features. This score gives no credit for partial recovery: If one interaction $S^\pm$ in $\mathscr S^*$ is $\{(1,+1), (2,+1)\}$, there will be no credit for $\mathscr S$ if it contains subsets of $S^\pm$ such as $\{(1,+1)\}$ or same features with different signs such as $\{(1,+1), (2,-1)\}$. While this score can be overly restrictive for practical problems, it is suitable for our simulation because we would like to evaluate whether LSSFind can consistently recover the interactions in the LSS model.
The simulation results are shown in Figure S1. 
In general, the performance of LSSFind sharply degrades when the number of basic interactions and the order of interactions increases.  
For $K=1$ and $L=2,3, 4$ LSSFind almost always recovers the correct basic signed interactions.
For $K = L = 2$ it mostly recovers the correct basic signed interactions, except for small SNR.
When $K=2$ and $L=3,4$, LSSFind rarely recover the basic signed interactions, for this simulation setup, resulting in a score of almost zero.
Note that this is consistent with our results in Theorem \ref{theo:mainResult}, which indicates the problem is much harder for more interactions and higher-order interactions.
We also explored the high-dimensional case. 
When $p=20, 50, 100, 200$ and $n=1000\cdot(1 + \log (p / 20)$, the score for LSSFind is shown in Figure \ref{fig:high_dimension}. 
The scaling of $p$ and $n$ is chosen to make sure $\log p / n \approx 0.001$ and also when $p=20$, $n$ will be $1000$, which corresponds to our previous numerical setting for better comparison. 
Recall that Theorem \ref{theo:mainResult} and \ref{Theo:mainResult_final} require condition $\log(p)/n\to 0$, as stated in Condition C4. 
We also note that $\log(p)/n\to 0$ is commonly imposed when analyzing lasso problems, too \cite{Lasso1996,zhao06a,hastie2015statistical}. 
As can be seen in Figure \ref{fig:high_dimension}, 
the score increases and approaches to 1 as the dimension $p$ increases.
This is consistent with Theorem \ref{Theo:mainResult_final} which shows that LSSFind can recover the underlying interactions for the high dimensional case.

\subsection{Robustness to LSS model violations}\label{subsec:sim_robust}

{
In the following, we present simulation results for LSSFind when the data is generated from a misspecified LSS model, that means, some of the LSS model assumptions are violated.
We find that LSSFind deteriorate when the LSS model is violated.
We consider a misspecified LSS model with $SNR = 5$ and 2 order-2 interactions with $p=20$ features and $n=1,000$ samples, analog as in Figure \ref{fig:Sim1}, second row, first column, third bar. 
We consider the following violations of LSS model assumptions: 
}
    \begin{itemize}
        \item \textbf{Overlapping interactions:} different basic interactions have overlapping features. When $\textrm{overlap}=1$, the basic interactions are $((1, -1), (2, -1))$, $( (2,-1)$, $(3, -1)$).
        \item \textbf{Correlated features:} Different features are correlated instead of independent. When $\text{corr}=\alpha$, the correlation between feature $j_1$ and $j_2$ is $\alpha^{|j_1 - j_2|}$.
        \item \textbf{Heavy-tail noise:} the noise follows a Laplace or Cauchy distribution, which have heavier tails than (sub-)Gaussian distributions. The noise is normalized such that the SNR is 50.
    \end{itemize}
Results of LSSFind are shown in Figure \ref{fig:LSSviolation}.
For heavy-tail noise, we observe a gradual drop in performance.
For the correlated feature case, one can see that 
LSSFind has reasonable performance when the correlation is close to zero but its performance deteriorates when the correlation is high.
Similar, for the overlapping feature case the performance worsens.

\subsection{Empirical comparison between LSSFind and iRF}\label{subsec:compareIRF}
Our original motivation to study DWP in RF tree ensemble came from the strong positive empirical evidence of iRF \cite{Basu2017,Kumbier2018}.
There are three major reasons why the full iRF procedure is hard to analyze theoretically:
First, the iterative re-weighting in iRF is based on the feature importance metric of \textit{mean decrease in impurity (MDI)}. 
Analyzing MDI for the RF algorithm is a challenging task on its own.
In particular, MDI of noisy features in deep trees are known to have a bias, \cite{zhou2020a,li2019,loecher2020unbiased}, which may propagate through various iterations in iRF and makes a theoretical analysis very challenging.
Second, the iRF procedure selects interactions from the paths of the RF tree ensemble via the RIT algorithm \cite{shah2014random}.
Thereby, individual paths are weighted according to the number of observations which fall into their respective leave nodes.
This means that the selected feature interactions of iRF cannot be derived from the RF tree ensemble directly, but depend on the data in a more complex way.
Third, the outer stability layer of iRF, where interactions are evaluated based on their consistent appearance among several bootstrap replications of the procedure, adds an additional layer of complexity for theoretical analysis.

In order to still analyze the major aspects of iRF theoretically, we proposed the related LSSFind algorithm.
Instead of iterative re-weighting via MDI, LSSFind introduces a single hard-threshold on the impurity index at individual tree nodes.
Moreover, instead of selecting interactions via RIT, LSSFind is based on DWP, which is derived from the tree ensemble directly, without an additional data dependent sampling scheme.
In other words,
although a high DWP in LSSFind does not exactly correspond to the RIT interaction selection strategy employed in iRF, they both build on similar high-level quantities, namely, sets of stable features which \textit{often} appear together on decision paths in an RF tree ensemble. 
Therefore, our theoretical results on DWP and LSSFind provide evidence that the general interaction discovery strategy of iRF is theoretically justified.
In the following, we complement our theoretical findings about LSSFind with an empirical comparison between iRF and LSSFind.

We consider the same simulation setting as in Section \ref{subsec:simSection}.
However, we replace the very strict performance measure in \eqref{eq:proxScore} by a weaker one\footnote{The results of iRF for the set-wise Jaccard distance in \eqref{eq:proxScore} are shown in Figure\ref{fig:LSSiRFcompare2}}.
Specifically, given a set $\mathscr S^*$ of $K$ true basic signed interactions from the respective LSS model and output from LSSFind and iRF, respectively, $\mathscr S$, we now evaluate their proximity based on:
\begin{align}\label{eq:proxScore2}
\overline{\mathrm{score}}(\mathscr S^*, \mathscr S) = \frac{|\{\cup_{S^- \in \mathscr S^*} S \} \cap \{\cup_{S^- \in \mathscr S} S \} |}{|\{\cup_{S^- \in \mathscr S^*} S \} \cup \{\cup_{S^- \in \mathscr S} S \} |}.
\end{align}
Note that \eqref{eq:proxScore2} corresponds to the Jaccard distance on the set of \emph{unsigned} features which appear in any of the detected interactions.
While the stricter metric in \eqref{eq:proxScore} is more appropriate to evaluate finite sample validity of Theorem \ref{theo:mainResult}, the relaxed version in \eqref{eq:proxScore2} is arguably of more practical interest. 
This is because it gives partial credit for interactions that are almost, but not perfectly, recovered. If $\overline{\mathrm{score}}(\mathscr S^*, \mathscr S)$ is high, it means the features in the discovered interactions overlap with the features in the true interactions, which would greatly narrow down the interaction search space and save tremendous effort for subsequent analysis for a practical problem.

For the iRF algorithm we used the signed iRF algorithm (siRF) from the Python iRF package \texttt{iRF}\footnote{https://github.com/Yu-Group/iterative-Random-Forest}, with default parameter settings and a threshold on iRF's stability score of 0.5 for interaction selection, as recommended in \cite{Kumbier2018}.
Simulation results are shown in Figure \ref{fig:LSSiRFcompare}.
When the LSS model as in \eqref{eq:simulations_model_lss} is relatively simple, for example, when it has only a single signed interaction (K=1) or only a single feature per signed interaction (L =1), iRF and LSSFind perform comparably (see fist row and second row first column of Figure \ref{fig:LSSiRFcompare}). 
However, when the LSS model gets more complex, with several additive interactions ($K > 1$) each having more than one signed features ($L > 1$) (see second row, second and third columns in Figure S3), iRF outperforms LSSFind in terms of the metric \eqref{eq:proxScore2}\footnote{In contrast, for the stricter performance metric \eqref{eq:proxScore}, which precisely captures the interaction detection property of Theorem \ref{theo:mainResult}, we note that LSSFind outperforms iRF, see Figure \ref{fig:LSSiRFcompare2}.}. 
In summary, we find that iRF outperforms LSSFind in situations where the underlying LSS model is more complex and when a flexible performance metric is chosen.
This appears to be consistent with the fact that the iRF algorithm has witnessed empirical success on specific domain data problems \citep{Basu2017}, whereas LSSFind was specifically constructed in such a way that it reflects our result in Theorem \ref{theo:mainResult}.

\section{Discussion} \label{sec:discussion}
Relevant statistics theory starts with a model that is a good approximation to reality. 
Thus, it is important to derive theoretical results under a model that is scientifically motivated. 
Our proposed LSS model class provides such a family that reflects the biological phenomena of biomelecule interacting through thresholding. 
Also, analyzing RF based algorithms under different models rather than the smoothness classes in the literature can give insights into their empirical adaptivity. 
Our results are the first to give a theoretical result that DWP of a set of features in an RF tree ensemble recovers high order interactions under the LSS model and reasonable conditions on the RF hyperparameters.
Moreover, the universality of interaction's DWP in LSS models gives insights into the general difference between quantitative (e.g., prediction accuracy) and qualitative (e.g., interaction recovery) information extraction. 
In scientific problems often the latter is of higher interest. 
Thus, this work narrows the gap between theory and practice for Boolean interaction discovery and is of general interest to the fields of statistics, data science, machine learning, and scientific fields such as genomics.

Our theoretical analysis also gives some insights of RF for tuning a crucial hyper-parameter $\mt$: 
Given an interaction with a fixed size, the non-interaction DWP upper bound in Theorem \ref{theo:mainResult} depends only on $C_m$ and $C_m$ is only constrained by $\mt$ (\ref{A:mtry}). Therefore, one can find an optimal $\mt$ that minimizes this upper bound. The optimal choice of $\mt$ turns out to be $\mt^\star = p \cdot ( 0.5  - s/(2(p-2))$. 
If one third of all features are signal features, that is, $s = p/3$, $\mt^\star$ recovers the default choice in standard RF implementations for regression, namely, $\mt^\star \approx p/3$.
However, when $p \gg s$, the optimal choice from our theoretical results corresponds to $\mt \approx p/2$, which suggests that with the presence of many noisy features, $\mt$ should be larger than $p/3$ as in the default choice. Further investigations through data inspired simulations and theoretical analyses are needed.

One might wonder whether the form of interaction defined by the LSS model constitutes a particularly \textit{difficult} or a particularly \textit{easy} form of feature interaction.
In general, there appears to be no clear (mathematical) answer to this question, as one cannot define what is meant by \textit{feature interaction} in a clear way for a generic (possibly discontinuous) regression function $f(X) = E(Y | X)$.
For example, it is easy to check that for any multivariate function $f: [0,1]^p -> \R$ one can find (possibly discontinuous) univariate functions $g, h_1,\ldots, h_p$, such that $f(x_1,\ldots, x_p) = g(h_1(x_1) + \ldots + h_p(x_p))$.  
We stress that the reason why we considered the LSS model in this work was not because it defines a particularly \textit{easy} form of interaction, but rather its biological relevance, as the thresholding relationships captured in the LSS model are observed in various biological data.

Although, the LSS model is motivated from biological phenomena, some of the assumptions which we made in order to derive our theoretical results might be difficult to justify directly in real-world-problems, in particular, the independence condition C1 and the non-overlapping interaction set condition C3.
Note, however, that in many application settings, it is possible to overcome these limitations by appropriate data preprocessing, e.g., decorrelating features (recall the discussion after condition C1).
Nevertheless, for future work it will be interesting to extend our results to a general LSS model (with possibly overlapping interaction sets and correlated features) or even interaction models beyond Boolean interactions, in order to further close the gap between theory and practice.

Finally, it will also be of interest to compare LSSFind and iRF with methods that, more generally, employ a ML black box model to extract interactions.
For example, when individual features are independent, as we assume in C1, one can use Monte Carlo methods \cite{saltelli2002} to estimate higher-order Sobol indices for the fitted ML model.

\section{Proof of Theorem \ref{theo:mainResult_general_upper_bound} and Theorem \ref{theo:mainResult}}
\label{sec:proofoftheorem}
\subsection{Proof of the population case -- desirable features}
Recall that there are three different sources of randomness:
\begin{enumerate}
    \item ($\cD$) the randomness of the data $\cD$,
    \item ($T$) the randomness of the tree $T$, given the data $\cD$,
    \item ($\cP$) the randomness of the randomly selected path $\cP$, given the tree $T$.
\end{enumerate}
Note that, although the random path $\cP$ depends on all three sources of randomness (the data randomness, the tree randomness, and the additional path randomness), when we condition on the tree $T$, then the random path $\cP$ is independent of the data $\cD$.
In the first part of the proof, we will only consider the last two sources of randomness, namely, from the random tree and from the randomly selected path on the tree.
Also, recall that we define $S_j^-, S_j^+ \subset [p]\times \{-1,+1\}$ as the features in $S_j \subset [p]$ with $-$ and $+$ sign, respectively, that is
\begin{align}
    S_j^- = \{ (k, -1) \; : \; k \in S_j\} \subset [p]\times \{-1,+1\},\\
    S_j^+ = \{ (k, +1) \; : \; k \in S_j\} \subset [p]\times \{-1,+1\}.
\end{align}

For each node $t$ in a tree $T$, define $\dot \parent^\pm(t)$ to be the set of signed features used by the parents of $t$ in $T$ and $\dot \parent(t)$ to be the corresponding (unsigned) features. 
For any feature $j$, $(j, -)$ and $(j, +)$ can appear together in $\dot \parent^\pm(t)$. Furthermore, let $\parent^\pm(t)$ be a subset of $\dot \parent^\pm(t)$ by only including the signed feature that corresponds to the first split of the feature if a feature appeared multiple times in the path. 
As a result, for any feature $j$, at most one of $(j, +)$ and $(j, -)$ can appear in $\parent^\pm(t)$. 
Define $\parent(t)$ to be the set of (unsigned) features in $\parent^\pm(t)$. Because $\dot \parent^\pm(t)$ and $\parent^\pm(t)$ only differ in terms of feature signs, they correspond to the same set of features, i.e., $\dot \parent(t) = \parent(t)$.
Conditioned on a tree $T$, at every node $t$ of $T$ we now define the set of \textit{desirable} features with respect to the LSS model as follows.
\begin{definition}[Desirable features]\label{Def:desirable_features}
Define the \textit{desirable feature set} $U(t)\subset [p]$ to be 
\begin{align}\label{eq:desirableFeatures}
    U(t) \triangleq \left\{k\in [p] \; \Big| \; \exists \; j\in [J]\text{ s.t. }k\in S_j, \; S_j^+\cap \parent^\pm(t) = \emptyset\;\text{and}\; (k, -1)\not\in \parent^\pm(t)\right\}.
\end{align}
\end{definition}
Note that the set of desirable features $U(t)$ at a node $t$ is only defined w.r.t. some particular LSS model. In particular, it depends on the basic signed interactions $S_1^-,\ldots, S_J^-$.
Hence, for a given tree $T$ with node $t$, $U(t)$ is an oracle set, which cannot be computed from data.
The way to think about $U(t)$ is that it corresponds exactly to those set of features which would yield some impurity decrease if the tree was grown by seeing the full data distribution $P(X,Y)$ and hence, making every split at the correct split point.
Moreover, denote $t_{\leaf}$ to be the leaf node of $\cP$ and we define $\cF$ to be the desirable signed features of $F(t_{\leaf})$. That is, the signed features $k_t$ where for the node $t$ on the path $\cP$ we have $k_t \in U(t)$, i.e., 
\begin{align}\label{eq:defPF}
\begin{aligned}
    \cF(\cP) \triangleq \{ (k_t, b_t) \in F(t_{\leaf})  \; \Big| \;  k_t \in U(t), \text{ $t_{\leaf}$ is leaf node of $\cP$ }\} \subset [p]\times \{-1,+1\}.
    \end{aligned}
\end{align}
For notation simplicity, we use $\cF$ as the shorthand of $\cF(\cP)$.

Further, we define the event $\Omega_0$ to be that the desirable features are exhausted at the leaf node:
\begin{align}\label{eq:Omega0}
    \Omega_0 \triangleq  \{ U(t_{\leaf}) = \emptyset \text{ for the leaf node $t_{\leaf}$ of $\cP$}\}.
\end{align}

With these definitions we get the following lemma.
\begin{lemma}\label{lem:irrLN} For the event $\Omega_0$ in \eqref{eq:Omega0} it holds true that
\begin{align}\label{eq:irrLN}
   \Omega_0 \subset \bigcap_{j \in [J]} \{ S_j^- \subset \cF\}\cup \{S_j^+ \cap \cF \neq \emptyset\},
\end{align}
with $ \{ S_j^- \subset \cF\}\cap \{S_j^+ \cap \cF \neq \emptyset\} = \emptyset$.
\end{lemma}
\begin{proof}
For an arbitrary interaction $j \in [J]$, it follows from the definition of $U(t)$ that $\Omega_0$ implies either $S_j^+ \cap \parent^\pm(t_{\leaf}) \neq \emptyset$ or $S_j^- \subset \parent^\pm(t_{\leaf})$. 
First, consider $S_j^+ \cap \parent^\pm(t_{\leaf}) \neq \emptyset$. 
Let $(k,+1) \in S_j^+ \cap \parent^\pm(t_{\leaf})$ be the signed feature in $S_j^+ \cap \parent^\pm(t_{\leaf})$ that appears first on the path. 
Then, because $\parent^\pm(t)$ only considers the signed features when they first appear in a path, we have that $(k, +1)$ was desirable and thus, $(k, +1) \in \cF$, i.e., $S_j^+ \cap \cF \neq \emptyset$.
Second, consider $S_j^- \subset \parent^\pm(t_{\leaf})$. Then for any $(k, -1) \in S_j^-$, by definition of $F(t)$ we have that no $S_j^+$ feature appeared on the path before $(k, -1)$ and hence, $(k, +1) \in \cF$, i.e., $S_j^- \subset \cF$.
Finally, recall that by definition of $\cF$ both conditions in  \eqref{eq:irrLN} can never happen at the same time.
\end{proof}
Moreover, define
\begin{align}\label{eq:invFeat}
   C_{\mathrm{root}}(\cD) \triangleq \min_{k \in \cup_{j = 1}^J S_j} P_T\left(\text{the root node of $T$ splits on feature }k \; \middle| \; \cD \right).
\end{align}
We state the population version of our main results below.
\begin{theorem}\label{theo:SappearNew} 
For all $\tilde{S}^\pm \subset [p]\times \{-1,+1\}$ with $\tilde{s} = |\tilde{S}^\pm|$ we have that almost surely
\begin{align}\label{eq:theoNew0}
    P_{\cP}\left( \tilde{S}^\pm \subset \cF \;\middle|\; T, \cD \right) \leq 0.5^{\tilde{s}}
\end{align}
and if $\tilde{S}^\pm$ is a union signed interaction as in Definition \ref{def:unionInter}
then almost surely
\begin{align}\label{eq:theoNew1}
         P_{\cP}\left(\tilde{S}^\pm \subset \cF \;\middle|\; T, \cD \right)   \geq 0.5^{\tilde{s}} - P_{\cP}\left( \Omega_0^c \;\middle|\; T, \cD \right).
\end{align}
Moreover, if $\tilde{S}^\pm$ is not a union signed interaction then almost surely
\begin{align}\label{eq:theoNew2}
       P_{(\cP, T)}\left(\tilde{S}^\pm \subset \cF \;\middle|\;  \cD  \right)  \leq 0.5^{\tilde{s}}(1 - C_{\mathrm{root}}(\cD) /2).
\end{align}
\end{theorem}

\begin{proof}[Proof of Theorem \ref{theo:SappearNew}]
Recall that the path $\cP$ corresponding to $\cF$ is selected in such a way: one starts at the root node $t_{\mathrm{root}}$ and then randomly follows the paths in the tree either to the plus ($+1$) or to the minus ($ -1$) direction with probability $0.5$.
Let $\cB$ denote a set of i.i.d.\ Bernoulli coin flips taking values $+1$ and $-1$ with equal probability $0.5$.
Assume that at every node in the tree we draw one of the Bernoulli coin flips $B \in \cB$ to decide whether we follow the path in the plus ($B = +1$) or in the minus ($B = -1$) direction.
In particular, for any feature $k \in [p]$, let $B^k \in \cB$ be the Bernoulli random variable we draw when $k$ appears for the first time on $\cP$.

\textbf{Proof of \eqref{eq:theoNew0}}:
Note that when $(k, -1) \in \cF$, $B^k = -1$.
Similar, when $(k, +1) \in \cF$, this implies that $B^k = +1$.
Consequently, for any $\tilde{S}^\pm = \{(k_1, b_1), \ldots, (k_{\tilde{s}}, b_{\tilde{s}}) \} \subset [p]\times \{-1,+1\}$ we have that
\begin{align}\label{eq:FsubsetB}
    \{ \tilde{S}^\pm \subset \cF \} \subset \{ B^{k_1} = b_1 \cap \ldots \cap  B^{k_{\tilde{s}}} = b_{\tilde{s}}  \}
\end{align}
and hence
\begin{align}
    P_{\cP}\left(\tilde{S}^\pm \subset \cF \middle|\; T, \cD \right) \leq P \left(B^{k_1} = b_1 \cap \ldots \cap  B^{k_{\tilde{s}}} = b_{\tilde{s}}  \right) = 0.5^{\tilde{s}}.
\end{align}
That completes the proof.

\textbf{Proof of \eqref{eq:theoNew1}}:
Consider any basic interaction $S_j = \{k_1, \ldots, k_{s_j}\}$, $j \in [J]$, then by Lemma \ref{lem:irrLN} we have that
\begin{align}\label{eq:BsubsetF_int}
   \Omega_0 \cap  \{ B^{k_1} = \ldots = B^{k_{s_j}} = -1 \} \subset  \{ S_j^- \subset \cF \}.
\end{align}
Moreover, when $s_j = 1$, we also have that
\begin{align}\label{eq:BsubsetF_single}
  \Omega_0 \cap  \{ B^{k_1} = +1 \} \subset  \{ S_j^+ \subset \cF \}.
\end{align}
Consequently, when $\tilde{S}$ is a union interaction as in Definition \ref{def:unionInter} it follows that 
\begin{align}\label{eq:BsubsetF}
  \Omega_0 \cap   \{B^{k_1} = b_1 \cap \ldots \cap  B^{k_{\tilde{s}}} = b_{\tilde{s}}  \} \subset  \{ \tilde{S}^\pm \subset \cF \},
\end{align}
which, shows \eqref{eq:theoNew1}.

\textbf{Proof of \eqref{eq:theoNew2}}:
Assume that $\tilde{S}^\pm$ is not a union interaction.
If any of the following is true:
\begin{itemize}
    \item $\tilde{S}^\pm$ contains any noisy signed feature $(k,b)$ that's not contained in $\cup_j S_j^+ \cup S_j^-$;
    \item for some signal feature $k \in \cup_j S_j$ we have that $(k, +1), (k, -1) \in \tilde{S}^\pm$;
    \item $|\tilde{S}^\pm \cap S_j^+| > 1$ for some $j \in [J]$; 
\end{itemize}
Then by definition of $U(t)$ in \eqref{eq:defPF}, $P_{\cP}\left( \tilde{S}^\pm \subset \cF \middle| \; T, \cD \right) = 0$ and thus, \eqref{eq:theoNew2} holds.

Thus, we can assume that $\tilde{S}^\pm$ contains no noisy features and there exists some interaction $j \in [J]$ with $s_j >1$ such that $(S_j^- \cup S_j^+) \cap \tilde{S}^\pm \neq \emptyset$ and for some $(k,-1) \in S_j^-$ we have that $(k,-1) \not\in \tilde{S}^\pm$.

First, assume that $(k, +1) \not\in \tilde{S}^\pm$.
Then, whenever $t_{\mathrm{root}}$ splits on feature $k$, we have that $\{\tilde{S}^\pm \subset \cF\}$ implies\footnote{Note that this requires the interactions to be disjoint, as otherwise the features in $\tilde{S}^\pm \cap (S_j^+ \cup S_j^-)$ may also appear in other interactions $S_l$ with $l \neq j$ and $k\not\in S_l$ and thus, even when $B^k = +1$ it is possible that $\tilde{S}^\pm \cap (S_j^+ \cup S_j^-) \subset \cF$.}
$B^k = -1$
and thus,
\begin{align*}
P_{(\cP, T)}\left(\tilde{S}^\pm \subset \cF | \cD\right) &= \sum_{\tilde{k} \in [p]}  P_{(\cP, T)}\left(\tilde{S}^\pm \subset \cF \cap t_{\mathrm{root}}\text{ splits on }\tilde{k} | \cD\right)\\
    &\leq \sum_{\tilde{k} \neq k}  P_{(\cP, T)}\left( B^{k_1} = b_1 \cap \ldots \cap  B^{k_{\tilde{s}}} = b_{\tilde{s}}  \cap t_{\mathrm{root}}\text{ splits on }\tilde{k} | \cD\right)\\
    &+ P_{(\cP, T)}\left(B^{k_1} = b_1 \cap \ldots \cap  B^{k_{\tilde{s}}} = b_{\tilde{s}}   \cap B^k = -1 \cap t_{\mathrm{root}}\text{ splits on }k | \cD\right)\\
    &= \sum_{\tilde{k} \neq k}  P\left( B^{k_1} = b_1 \cap \ldots \cap  B^{k_{\tilde{s}}} = b_{\tilde{s}}   \right) P_{T}\left( t_{\mathrm{root}}\text{ splits on }\tilde{k} | \cD\right)\\
    &+ P\left( B^{k_1} = b_1 \cap \ldots \cap  B^{k_{\tilde{s}}} = b_{\tilde{s}}  \cap B^k = -1 \right) P_{ T}\left( t_{\mathrm{root}}\text{ splits on }k| \cD \right)\\
    &= 0.5^{\tilde{s}} (1 - P_{ T}(t_{\mathrm{root}}\text{ splits on }k| \cD)) + 0.5^{\tilde{s} + 1}P_{ T}(t_{\mathrm{root}}\text{ splits on }k| \cD)\\
    & \leq 0.5^{\tilde{s}}(1 - C_{\mathrm{root}}/2),
\end{align*}
where we made use of the fact that the tree $T$ is independent of the Bernoulli random variables $\cB$.

Second, assume that $(k, +1) \in \tilde{S}^\pm$.
If $\tilde{S}^\pm \cap S_j^-  \neq \emptyset$, then $\{\tilde{S}^\pm \subset \cF\}$ implies\footnote{Again, this requires the interactions to be disjoint, as otherwise the features in $\tilde{S}^\pm \cap (S_j^- \cup S_j^+) \setminus {(k, +1)}$ may also appear in other interactions $S_l$ with $l \neq j$ and thus, even when $t_{\mathrm{root}}$ splits on $k$ with $B^k = +1$, it is possible that $\tilde{S}^\pm \cap (S_j^- \cup S_j^+) \setminus {(k, +1)} \subset \cF$.} that  $t_{\mathrm{root}}$ does not split on $k$ and thus 
\begin{align*}
P_{(\cP, T)}\left(\tilde{S}^\pm \subset \cF \middle|\; \cD \right) &\leq P\left( B^{k_1} = b_1 \cap \ldots \cap  B^{k_{\tilde{s}}} = b_{\tilde{s}}  \right) P_{ T} \left(t_{\mathrm{root}}\text{ does not split on }k \middle| \; \cD \right)\\ 
&= 0.5^{\tilde{s}} P_{ T} \left(t_{\mathrm{root}}\text{ does not split on }k  \middle|\; \cD \right) \leq 0.5^{\tilde{s}}(1  -  C_{\mathrm{root}}).
\end{align*}
If $\tilde{S}^\pm \cap S_j^-  = \emptyset$, let $k^\star \in S_j$ and $k^\star \neq k$.
Because $(k, +1) \in \tilde{S}^\pm$, we can assume that $(k^{\star }, +1) \not\in \tilde{S}^\pm$; otherwise $|\tilde{S}^\pm \cap S_j^+| > 1$, which implies $ P\left(\tilde{S}^\pm \subset \cF\right) = 0$.
When $t_{\mathrm{root}}$ splits on $k^\star$, $\{\tilde{S}^\pm \subset \cF\}$ implies\footnote{Again, this requires the interactions to be disjoint.} $B^{k^\star} = -1$ and thus,
\begin{align*}
    &P_{(\cP, T)}\left(\tilde{S}^\pm \subset \cF \middle| \; \cD \right) = \sum_{\tilde{k} \in [p]}  P_{(\cP, T)} \left(\tilde{S}^\pm \subset \cF \cap t_{\mathrm{root}}\text{ splits on }\tilde{k} \middle| \; \cD \right)\\
     &\leq \sum_{\tilde{k} \neq k^\star}  P_{(\cP, T)} \left( B^{k_1} = b_1 \cap \ldots \cap  B^{k_{\tilde{s}}} = b_{\tilde{s}}  \cap t_{\mathrm{root}}\text{ splits on }\tilde{k} \middle| \; \cD \right)\\
    &+ P_{(\cP, T)} \left( B^{k_1} = b_1 \cap \ldots \cap  B^{k_{\tilde{s}}} = b_{\tilde{s}}  \cap B^{k^\star} = -1 \cap t_{\mathrm{root}}\text{ splits on }k^\star \middle| \; \cD \right)\\
    &\leq  0.5^{\tilde{s}} (1 - P_{ T} \left(t_{\mathrm{root}}\text{ splits on }k^\star \middle| \; \cD \right) ) + 0.5^{\tilde{s} + 1} P_{(\cP, T)} \left(t_{\mathrm{root}}\text{ splits on }k^\star \middle| \; \cD \right)\\ &\leq 0.5^{\tilde{s}}(1 - C_{\mathrm{root}}/2).
\end{align*}
Thus, we have shown \eqref{eq:theoNew2}.
\end{proof}

\subsection{Proof of the finite sample case}\label{subsec:finiteSampleProof}

\subsubsection{Filtering of desirable features and impurity}\label{subsubsec:filtering}

Recall that $R_{t,l} = R_t\cap \{X|X_{k_t}\leq \gamma_{t}\}$ and $R_{t,r} = R_t\cap \{X|X_{k_t} > \gamma_{t}\}$ denote the region corresponding to the left and right children for node $t$. In other words, node $t$ divides the region $R_t$ into $R_{t,l}$ and $R_{t,r}$. Recall that $N_{n}(t)$ is the number of samples in the region $R_t$, i.e., $N_n(t) = \sum_{i=1}^n \1(x_i\in R_t)$. We will use an equivalent formula for the impurity as in Lemma \ref{lemma:two_impurities}.

\begin{lemma}\label{lemma:two_impurities}
$\Delta_I^n(R_{t,l}, R_{t,r})$ defined in \eqref{eq:impdecrease} in the main paper is equivalent to \eqref{eq:impInc}:
\begin{align}\label{eq:impInc}
    \Delta_I^n(R_{t,l}, R_{t,r}) = \frac{N_n(t_l)N_n(t_r)}{n(N_n(t_l) + N_n(t_r))}\left(\frac{1}{N_n(t_l)}\sum_{\bx_i \in {R_{t,l}}}y_i - \frac{1}{N_n(t_r)}\sum_{\bx_i \in R_{t,r}}y_i\right)^2.
\end{align}
\end{lemma}
\begin{proof}
We have that
\begin{align*}
&\Delta_I^n(R_{t,l}, R_{t,r})\\
= & \frac{1}{n}\left(\sum_{\bx_i\in R_t}(y_i - \frac{1}{N_n(t)}\sum_{\bx_i \in R_t}y_i)^2 - \sum_{\bx_i\in R_{t,l}}(y_i - \frac{1}{N_n(t_l)}\sum_{\bx_i \in R_{t,l}}y_i)^2 - \sum_{\bx_i\in R_{t,r}}(y_i - \frac{1}{N_n(t_r)}\sum_{\bx_i \in R_{t,r}}y_i)^2\right)\\
= & \frac{1}{n}\left(\sum_{\bx_i\in R_t}y_i^2 - \frac{1}{N_n(t)}(\sum_{\bx_i \in R_t}y_i)^2 - \sum_{\bx_i\in R_{t,l}}y_i^2 + \frac{1}{N_n(t_l)}(\sum_{\bx_i \in R_{t,l}}y_i)^2 - \sum_{\bx_i\in R_{t,r}}y_i^2 + \frac{1}{N_n(t_r)}(\sum_{\bx_i \in R_{t,r}}y_i)^2\right)\\
= & \frac{1}{n}\left( - \frac{1}{N_n(t)}(\sum_{\bx_i \in R_t}y_i)^2 + \frac{1}{N_n(t_l)}(\sum_{\bx_i \in R_{t,l}}y_i)^2 + \frac{1}{N_n(t_r)}(\sum_{\bx_i \in R_{t,r}}y_i)^2\right).
\end{align*}
If we denote $A = \sum_{\bx_i \in R_{t,l}}y_i$ and $B = \sum_{\bx_i \in R_{t,r}}y_i$, the above formula is the same as :
\begin{align*}
     & \frac{1}{n}\left( - \frac{1}{N_n(t)}(A+B)^2 + \frac{1}{N_n(t_l)}A^2 + \frac{1}{N_n(t_r)}B^2\right)\\
     = &\frac{1}{n}\left( \frac{N_n(t_r)}{N_n(t_l)N_n(t)}A^2 + \frac{N_n(t_l)}{N_n(t_r)N_n(t)}B^2 - \frac{2}{N_n(t)}AB \right)\\
     = &\frac{N_n(t_l) N_n(t_r)}{n N_n(t)}\left( \frac{1}{N_n(t_l)^2}A^2 + \frac{1}{N_n(t_r)^2}B^2 - \frac{2}{N_n(t_l)N_n(t_r)}AB \right) = \frac{N_n(t_l) N_n(t_r)}{n N_n(t)} \left(\frac{1}{N_n(t_l)} A - \frac{1}{N_n(t_r)} B\right)^2.
\end{align*}
\end{proof}
Let $\mathscr R$ denote the set of axis-aligned hyper-rectangles obtained by splitting the unit hyper-rectangle consecutively, where each split satisfies assumption \ref{A:balancedsplit} in the main text. We study $\mathscr R$ because it contains all the possible rectangles that can represent region of a node in a tree.
Let $\mathscr{R}_d$ be the set of rectangles obtained by splitting the unit hyper-rectangle $d$ times, where each split satisfies assumption \ref{A:balancedsplit} from the main text.
Then $\mathscr{R} = \cup_{d\geq 1} \mathscr{R}_d$, and for any $R\in\mathscr{R}_d$, we have $\mu(R) \leq (1-C_{\gamma})^{d}$ (recall that $\mu(R)$ denotes the volume of $R$). For any region $R$, we denote $N_R$ to be the number of points in $R$. 
\begin{lemma}
\label{lemma:generalization_small_rectangles}
Suppose that assumption \ref{A:balancedsplit} from the main text is satisfied. Then for any $d\geq 1$ it holds true that
\[
\max_{R\in \cup_{d_1>d}\mathscr{R}_{d_1}}\Bigg|\frac{1}{n}\sum_{i=1}^ny_i\1(\bx_i\in R_1) - \mathbb E (Y\cdot\1(X\in R))\Bigg|\leq C_Y\left(\max_{R\in \mathscr{R}_d} \Big|\frac{N_R}{n} - \mu(R)\Big|\right) + 2C_Y (1-C_{\gamma})^d.
\]
\end{lemma}
\begin{proof} [Proof of Lemma]
For any $R_1\in \mathscr{R}_{d_1}$, $d_1>d$, there exists $R_0\in\mathscr{R}_d$ such that
$R_1\subset R_0$. Therefore, $N_{R_1} < N_{R_0}$ and
\[
\Big|\frac{1}{n}\sum_{i=1}^ny_i\1({\bx_i\in R_1})\Big| \leq \frac{N_{R_1}}{n}C_Y < \frac{N_{R_0}}{n}C_Y.
\]
Since $R_0\in\mathscr{R}_d$, we have
\begin{equation}
    \label{eq:rectangle_size_bound}
    \frac{N_{R_0}}{n} \leq \max_{R\in \mathscr{R}_d} \Big|\frac{N_R}{n} - \mu(R)\Big| + \max_{R\in \mathscr{R}_d} \mu(R) \leq \max_{R\in \mathscr{R}_d} \Big|\frac{N_R}{n} - \mu(R)\Big| + (1-C_{\gamma})^{d}.
\end{equation}
Therefore,
\begin{align*}
&\Bigg|\frac{1}{n}\sum_{i=1}^ny_i\1({\bx_i\in R_1}) - \mathbb E (Y\cdot\1({X\in R}))\Bigg| \\
\leq \;\ &\Big|\frac{1}{n}\sum_{i=1}^n y_i\1({\bx_i\in R_1})\Big| + |\mathbb E (Y\cdot\1({X\in R}))| \\
\leq \;\ & \frac{N_{R_0}}{n}C_Y + C_Y(1-C_{\gamma})^{d+1}\\
\leq \;\ & C_Y\left(\max_{R\in \mathscr{R}_d} \Big|\frac{N_R}{n} - \mu(R)\Big| + (1-C_{\gamma})^{d}\right) + C_Y(1-C_{\gamma})^{d+1}\\
\leq \;\ & C_Y\left(\max_{R\in \mathscr{R}_d} \Big|\frac{N_R}{n} - \mu(R)\Big|\right) + 2C_Y (1-C_{\gamma})^d.
\end{align*}
Since $R_1$ is arbitrary, we have
\begin{equation}
    \max_{R\in \cup_{d_1>d}\mathscr{R}_{d_1}}\Bigg|\frac{1}{n}\sum_{i=1}^n  y_i\1({\bx_i\in R_1}) - \mathbb E (Y\cdot\1({X\in R}))\Bigg|\leq C_Y\left(\max_{R\in \mathscr{R}_d} \Big|\frac{N_R}{n} - \mu(R)\Big|\right) + 2C_Y (1-C_{\gamma})^d.
\end{equation}
\end{proof}

\begin{proposition}\label{prop:uniform_c}
Suppose that constraint C4 and
assumption \ref{A:balancedsplit} from the main text hold true. Then
\[
\max_{R\in \mathscr{R}} \Big|\frac{N_R}{n} - \mu(R)\Big| \pto 0,
\]
and
\[
\max_{R\in \mathscr{R}} \Big|\frac{1}{n}\sum_{i=1}^n y_i\1({\bx_i\in R}) - \mathbb E (Y\cdot\1({X\in R}))\Big| \pto 0.
\]
\end{proposition}
\begin{proof}
For any fixed $d$, let $G_n(\mathscr{R}_{d})$ be the growth function for the set of rectangles $\mathscr{R}_{d}$ defined in Chapter 5.2 of Vapnik \cite{vapnik1998statistical}, i.e., 

\[
G_n(\mathscr{R}_{d}) \triangleq \max_{\bx_i\in \R^p,y_i\in \R} \log \left|\left\{(\1({y_1 \geq \theta, \bx_1\in R}), \ldots, \1({y_n \geq \theta, \bx_n\in R}))\Big| R\in \mathscr R_d, \theta\in \R \right\}\right|.
\]
Here for any set $A$, $|A|$ denotes the number of elements in $A$.

We claim that $G_n(\mathscr{R}_{d})\leq \log(n(2np)^d)$. This is because at each of $d$ splits, we have at most $p$ directions and at most $n$ split points to choose from. Therefore, splitting $d$ times can create no more than $(2np)^d$ different separations of the $n$ data points. Furthermore, within each rectangle, the indicator functions $\1({y_i\geq\theta}), \theta\in \mathbb{R}$ can at most create $n$ separations.

Thus,
\begin{equation}\label{eq:growth_function_bound}
G_n(\cup_{d_0\leq d}\mathscr{R}_{d_0})\leq \log (d \exp(G_n(\mathscr R_{d}))) \leq \log(nd(2np)^d),
\end{equation}
and 
\[
\frac{G_n(\cup_{d_0\leq d}\mathscr{R}_{d_0})}{n} \leq \frac{\log(nd) + d\log(2n)}{n} + \frac{d\log p}{n} \to 0.
\]
Therefore, by Theorem 5.1 of \cite{vapnik1998statistical}:

\textbf{Theorem 5.1 in \cite{vapnik1998statistical}}: Let $A \leq Q(z, \alpha) \leq B$, $\alpha\in \Lambda$ be a measurable set of bounded real-valued functions. Let $G_n$ be the growth function of the indicator functions induced by $Q$, then we have the following inequality:
\[
P\left\{\sup_{\alpha\in \Lambda} \left(\int Q(z, \alpha) dF(z) - \frac{1}{n}\sum_{i=1}^n Q(z_i, \alpha)\right)> \epsilon\right\}\leq 4 \exp\left\{\left(\frac{G_{2n}}{n} - \frac{(\epsilon - n^{-1})^2}{(B - A)^2}\right)n\right\}.
\]
we have
\begin{equation}\label{Eq:small_recs}
\max_{R\in \cup_{d_0\leq d}\mathscr{R}_{d_0}}\Bigg|\frac{1}{n}\sum_{i=1}^n y_i\1({\bx_i\in R}) - \mathbb E (Y\cdot\1({X\in R}))\Bigg| \pto 0.
\end{equation}
Taking $Y=1$ and we have 
\begin{equation}
    \max_{R\in \cup_{d_0\leq d}\mathscr{R}_{d_0}} \Big|\frac{N_R}{n} - \mu(R)\Big| \pto 0.
\end{equation}
By Lemma \ref{lemma:generalization_small_rectangles} and the above equation,
we have
\begin{align}
\label{eq:large_recs}
\begin{aligned}
    &\max_{R\in \cup_{d_1>d}\mathscr{R}_{d_1}}\Bigg|\frac{1}{n}\sum_{i=1}^n y_i\1({\bx_i\in R_1}) - \mathbb E (f(X)\cdot\1({X\in R_1}))\Bigg|\\
    \leq & \; C_Y\left(\max_{R\in \mathscr{R}_d} \Big|\frac{N_R}{n} - \mu(R)\Big|\right) + 2C_Y (1-C_{\gamma})^d \overset{p}{\leq} 3 C_Y (1 - C_\gamma)^d.
\end{aligned}
\end{align}
Since that holds for any fixed $d > 0$, we know the left hand side of \eqref{eq:large_recs} converges to zero in probability. Combining \eqref{Eq:small_recs} and \eqref{eq:large_recs}, we have shown that:
\[
\max_{R\in \mathscr{R}} \Big|\frac{1}{n}\sum_{i=1}^n y_i\1({\bx_i\in R}) - \mathbb E (Y\cdot\1({X\in R}))\Big| \pto 0.
\]
Since this holds for any bounded random variable $Y$, we can take $Y=1$ and we have shown
\begin{equation}
    \max_{R\in \mathscr{R}} \Big|\frac{N_R}{n} - \mu(R)\Big|  \pto 0.
\end{equation}
That completes the proof.
\end{proof}

\begin{proposition}[Subgaussian case]\label{prop:sub_gaussian}
Suppose that assumption \ref{A:balancedsplit} from the main text holds true and\\ $(\log n)^{1+\delta} \log p / n \to 0$ for some $\delta > 0$. Suppose $Y = \E(Y|X) + Z$ where $Z$ is independent of $X$ and $1$-subgaussian. 
Then
\[
\max_{R\in \mathscr{R}} \Big|\frac{1}{n}\sum_{i=1}^n y_i\1({\bx_i\in R}) - \mathbb E (Y\cdot\1({X\in R}))\Big| \pto 0.
\]
\end{proposition}
\begin{proof}
Denote $f(X) = E(Y|X)$ and $C_Y = \sum_{j=0}^{J}|\beta_j|$. Then $|f(X)|\leq C_Y$. 
Note that
\begin{align*}
&\max_{R\in \mathscr{R}} \Big|\frac{1}{n}\sum_{i=1}^n y_i\1({\bx_i\in R}) - \mathbb E (Y\cdot\1({X\in R}))\Big| \\
\leq \;\ & \max_{R\in \mathscr{R}} \Big|\frac{1}{n}\sum_{i=1}^n f(\bx_i)\1({\bx_i\in R}) - \mathbb E (f(X)\cdot\1({X\in R}))\Big| + \max_{R\in \mathscr{R}} \Big|\frac{1}{n}\sum_{i=1}^n z_i\1({\bx_i\in R}) \Big|.
\end{align*}
Here $z_i = y_i - f(\bx_i)$ represents the noise terms.
Our proof proceeds in the following two steps.

\textbf{Step 1}.
Show that
\begin{equation}
\label{eq:expectation_generalization}
    \max_{R\in \mathscr{R}} \Big|\frac{1}{n}\sum_{i=1}^n f(\bx_i)\1({\bx_i\in R}) - \mathbb E (f(X)\cdot\1({X\in R}))\Big| \pto 0.
\end{equation}
Step 1 is similar to the proof of Proposition \ref{prop:uniform_c} but the difference is that we need the convergence rate.
Let 
\[
\delta_0 = \frac{\delta}{2\delta + 4},
\]
and take \[
d = \left(\frac{n}{\log (np)}\right)^{\delta_0} \to \infty.
\]
Let $G_n(\mathscr{R}_{d})$ be the growth function for the set of rectangles $\mathscr{R}_{d}$. 
By \eqref{eq:growth_function_bound}, we have
\[
\frac{G_n(\cup_{d_0\leq d}\mathscr{R}_{d_0})}{n} \leq \frac{\log(nd) + d\log(2n)}{n} + \frac{d\log p}{n} = O\left(\frac{d\log (np)}{n}\right) = O\left(\left(\frac{\log(np)}{n}\right)^{1-\delta_0}\right)\to 0.
\]
Therefore, by Theorem 5.1 of \cite{vapnik1998statistical}, we have
\begin{equation}
\label{eq:small_rectangles}
\max_{R\in \cup_{d_0\leq d}\mathscr{R}_{d_0}}\Bigg|\frac{1}{n}\sum_{i=1}^n f(\bx_i)\1({\bx_i\in R_1}) - \mathbb E (f(X)\cdot\1({X\in R}))\Bigg| = o\left(\left(\frac{\log(np)}{n}\right)^{1/2-\delta_0}\right)\pto 0.
\end{equation}
Since this holds for any bounded random variable $Y$, we can take $Y=1$ and it follows that
\begin{equation}
    \label{eq:rectangle_size_deviation}
    \max_{R\in \mathscr{R}_d} \Big|\frac{N_R}{n} - \mu(R)\Big| \leq \max_{R\in \cup_{d_0\leq d}\mathscr{R}_{d_0}} \Big|\frac{N_R}{n} - \mu(R)\Big| = o\left(\left(\frac{\log(np)}{n}\right)^{1/2-\delta_0}\right) \pto 0.
\end{equation}
Since $d\to\infty$, $(1-C_\gamma)^d\to 0$. Therefore, by Lemma \ref{lemma:generalization_small_rectangles},
we have
\begin{equation}
\label{eq:large_rectangles}
\max_{R\in \cup_{d_1>d}\mathscr{R}_{d_1}}\Bigg|\frac{1}{n}\sum_{i=1}^n f(\bx_i)\1({\bx_i\in R_1}) - \mathbb E (f(X)\cdot\1({X\in R}))\Bigg| \pto 0.
\end{equation}
Combining \eqref{eq:small_rectangles} and \eqref{eq:large_rectangles}, \eqref{eq:expectation_generalization} is proved.

\textbf{Step 2}.
Show that 
\begin{equation}
\label{eq:noise_generalization}
    \max_{R\in \mathscr{R}} \Big|\frac{1}{n}\sum_{i=1}^n z_i\1({\bx_i\in R}) \Big| \pto 0.
\end{equation}

Note that 
\[
\max_{R\in \mathscr{R}} \Big|\frac{1}{n}\sum_{i=1}^n z_i\1({\bx_i\in R}) \Big| = \max\left\{\max_{R\in \cup_{d_0\leq d}\mathscr{R}_{d_0}} \Big|\frac{1}{n}\sum_{i=1}^n z_i\1({\bx_i\in R}) \Big|, \max_{R\in R\in \cup_{d_1>d}\mathscr{R}_{d_1}} \Big|\frac{1}{n}\sum_{i=1}^n z_i\1({\bx_i\in R}) \Big| \right\}.
\]
Therefore, it suffices to prove that both of the two terms on the right hand size converges to 0 in probability.
We begin with the first term: $\max_{R\in \cup_{d_0\leq d}\mathscr{R}_{d_0}} \Big|\frac{1}{n}\sum_{i=1}^n z_i\1({\bx_i\in R}) \Big|$. Since $X$ and $Z$ are independent and $Z$ is 1-subgaussian, by Hoeffding inequality,
\[
P\left(\Big|\frac{1}{n}\sum_{i=1}^n z_i\1({\bx_i\in R}) \Big|\geq \epsilon/2  \; \bigg| \; X\right) = P\left( \Big|\frac{1}{n}\sum_{i=1}^{N_R} z_i\Big|\geq \epsilon/2\right) \leq 2\exp\left(-\frac{n^2\epsilon^2}{8N_R}\right) 
\]
for any rectangle $R$.
Therefore by union bound,
\begin{align*}
    &P\left(\max_{R\in \cup_{d_0\leq d}\mathscr{R}_{d_0}}\Big|\frac{1}{n}\sum_{i=1}^n z_i\1({\bx_i\in R}) \Big|\geq \epsilon/2 \; \bigg| \;  X\right)\\
    \leq\;  &2\exp(G_n(\cup_{d_0\leq d}\mathscr{R}_{d_0}))\exp\left(-\frac{n\epsilon^2}{8}\right) \\
    \leq\; &2\exp\left(\log(nd(2np)^d) -\frac{n\epsilon^2}{8}\right)\to 0
\end{align*}
for any $\epsilon > 0$. Since the above upper bound on the probability is independent of $X$, we conclude that 
\[
\max_{R\in \cup_{d_0\leq d}\mathscr{R}_{d_0}}\Big|\frac{1}{n}\sum_{i=1}^n z_i\1({\bx_i\in R}) \Big| \pto 0.
\]
We now turn to the second term 
$\max_{R\in R\in \cup_{d_1>d}\mathscr{R}_{d_1}} \Big|\frac{1}{n}\sum_{i=1}^n z_i\1({\bx_i\in R}) \Big|$.
Let $\mathscr{R}^{s_0}$ be the set of rectangles with at most $s_0 = n/(\log n)^{1/2 + \delta_0}$ samples, then $\log |\mathscr{R}^{s_0}| \leq (s_0 + 1)\log n$.
By union bound,
\begin{align*}
    &P\left(\max_{R\in \mathscr{R}^{s_0}}\Big|\frac{1}{n}\sum_{i=1}^n z_i\1({\bx_i\in R}) \Big|\geq \epsilon/2 \; \bigg| \;  X\right)\\
    \leq\;  & 2\exp(\log |\mathscr{R}^{s_0}|)\exp\left(-\frac{n\epsilon^2}{8}\right)\\ \leq \; & 2\exp\left((s_0 + 1)\log n -\frac{n^2\epsilon^2}{8s_0}\right)\to 0.
\end{align*}
Therefore,
\[
\max_{R\in \mathscr{R}^{s_0}}\Big|\frac{1}{n}\sum_{i=1}^n z_i\1({\bx_i\in R}) \Big| \to 0.
\]
Hence, to prove \eqref{eq:noise_generalization}, it suffices to show that
$\cup_{d_1>d}\mathscr{R}_{d_1} \subset \mathscr{R}^{s_0}$
with probability tending to 1. 
Note that by definition of $\delta_0$, $\frac{1/2 + \delta_0}{1/2 - \delta_0} = 1 + \delta$. Therefore
\[
\left(\frac{\log(np)}{n}\right)^{\frac 12-\delta_0}(\log n)^{\frac 12 + \delta_0} = \left(\frac{\log(np)(\log n)^{1+\delta}}{n}\right)^{\frac 12-\delta_0} = \left(\frac{(\log n)^{2+\delta} + \log p(\log n)^{1+\delta}}{n}\right)^{\frac 12-\delta_0} \to 0.
\]
By \eqref{eq:rectangle_size_bound} and \eqref{eq:rectangle_size_deviation} we have
\[
\max_{R\in \cup_{d_1>d}\mathscr{R}_{d_1}} N_R \leq 
\max_{R\in \mathscr{R}_{d}} N_R = o\left(n\left(\frac{\log(np)}{n}\right)^{1/2-\delta_0}\right) = o(s_0).
\]
Therefore, $\max_{R\in \cup_{d_1>d}\mathscr{R}_{d_1}} N_R\leq s_0$ with probability tending to 1. The proof is now complete.
\end{proof}
Define population impurity decrease $\Delta_I(t)$ at a node $t$ to be
\begin{align}\label{eq:posMdi}
       \Delta_I(t) = \va(Y | R_t )- \frac{\mu(R_{t_l})}{\mu(R_t)} \va(Y | R_{t_l}) - \frac{\mu(R_{t_r})}{\mu(R_t)} \va(Y | R_{t_r}).
    \end{align}
Similar to Lemma \ref{lemma:two_impurities}, we know it is equivalent to:
\begin{align} \label{Eq:equivalent_impurity_decrease}
    \Delta_I(R_{t,l}(\gamma; k), R_{t,r}(\gamma; k)) = \frac{\mu(R_{t,l}(\gamma; k))\mu(R_{t,r}(\gamma; k))}{\mu(R_t(\gamma; k))}\Big[\mathbb E (Y|{X\in R_{t,l}(\gamma; k)}) - \mathbb E (Y|{X\in R_{t,r}(\gamma; k)})\Big]^2.
\end{align}
The following proposition shows that the finite-sample impurity decrease converges to the population impurity decrease uniformly.
\begin{proposition}\label{Prop:uniform_convergence} Suppose that 
constraint C4
and assumption \ref{A:balancedsplit} from the main text are satisfied. Then, we have the following two uniform convergence results:
\begin{enumerate}
    \item[a.] $\max_{R\in \mathscr{R}} \Big|\frac{N_R}{n} - \mu(R)\Big|\pto 0$,  
    \item[b.]$\underset{R_{t,l}, R_{t,r}\in \mathscr R}{\sup} \Big| \Delta_I^n(R_{t,l}, R_{t,r}) - \Delta_I(R_{t,l}, R_{t,r})\Big|\pto 0$. 
\end{enumerate}
\end{proposition}
\begin{proof}
\textbf{a.} This follow directly from Proposition \ref{prop:uniform_c}.

\textbf{b.} Let $f(x_1, x_2, y_1, y_2) = \frac{x_1x_2}{x_1+x_2}(y_1-y_2)^2$. Then $f$ is a Lipschitz function on $[0, 1]\times[0, 1]\times[-C_Y-1, C_Y+1]\times[-C_Y-1, C_Y+1]$.
Use the fact that 
$\max_{R\in \mathscr{R}} \Big|\frac{1}{n}\sum_{i=1}^n y_i\1({\bx_i\in R}) - \mathbb E (Y\cdot\1({X\in R}))\Big| \pto 0$ in Proposition \ref{prop:uniform_c} and the fact $\max_{R\in \mathscr{R}} \Big|\frac{N_R}{n} - \mu(R)\Big|\pto 0$ in a., by the continuous mapping theorem, we have
\[
\underset{R_{t,l}, R_{t,r}\in \mathscr R}{\sup} \Big| \Delta_I^n(R_{t,l}, R_{t,r}) - \Delta_I(R_{t,l}, R_{t,r})\Big|\pto 0.
\]
\end{proof}

Now we analyze the impurity decrease at each node of a tree. We consider three families of trees: $\mathcal{T}_0$, $\mathcal T_1$ and $\mathcal T_2$:
\[
\mathcal T_0 \triangleq \{\text{Any tree that satisfies \ref{A:balancedsplit}}\}.
\]

\[
\mathcal T_1 \triangleq \{\text{Any CART tree that satisfies \ref{A:balancedsplit} and \ref{A:no_bootstrap}}\}.
\]
\[
\mathcal T_2 \triangleq \{\text{Any CART tree that satisfies \ref{A:balancedsplit}, \ref{A:no_bootstrap}, and \ref{A:mtry}}\}.
\]
$\mathcal T_1$ is the family of CART trees that satisfy our assumptions but $\mt$ can be arbitrary. $\mathcal T_1$ is more restricted than $\mathcal T_0$ in the sense that the threshold $\gamma_t$ of any node $t$ of any tree in $\mathcal T_1$ must maximize the finite sample impurity decrease in \eqref{eq:impdecrease}. Thus, $\mathcal T_1$ depends on the data. 
For any $T\in \mathcal T_0$ and any $t\in T$ such that $\mathring U(t)\neq \emptyset$, its region $R_t$ is a rectangle:
\begin{align}
    R_t = \{x\in \mathbb R^p| \forall \ell\in [p], c_{low, \ell} < x_\ell \leq c_{high, \ell}\}.
\end{align}
where $c_{low, \ell}, c_{high,\ell} \in [0,1]$. 

By the definition of desirable feature set $U(t)$ in \eqref{eq:desirableFeatures}, we have its equivalent formula:
\[
U(t) \triangleq \cup_{j\in [J]:S_j^+\cap \parent^\pm(t) = \emptyset}S_j / \parent(t).
\] 
Define the set of noisy features to be its complement: $[p] / U(t)$. We also define 
\[
\mathring U(t) \triangleq \cup_{j\in [J]:S_j^+\cap \dot \parent^\pm(t) = \emptyset}S_j / \parent(t).
\]
Since $\parent^\pm(t) \subset \dot \parent^\pm(t)$, $\mathring U(t)\subset U(t)$. For any $\gamma$, denote $R_{t,l}(\gamma; k) = R_t\cap \{X| X_k\leq \gamma\}$ and $R_{t,r}(\gamma; k) = R_t\cap \{X| X_k > \gamma\}$.
First, for any node $t\in T$ and any $k\in \mathring U(t)$, we have a characterization for the impurity decrease:
\begin{lemma}\label{Lemma:impurity_decrease} For any $T\in \mathcal T_0$, $t\in T$, $j\in[J]$, $k\in S_j\cap U(t)$, and $\gamma \in (0,1)$, 
\begin{align*}
& \Delta_I(R_{t,l}(\gamma;k), R_{t,r}(\gamma;k)) \\
=& \mu(R_t)\cdot \beta_{j}^2P(\forall \; \ell\in S_j/\{k\},\; X_\ell \leq \gamma_{\ell}| X\in R_t)^2\cdot 
\Big(\1(\gamma \leq \gamma_{k})\cdot \frac{(1 - \gamma_{k})^2\gamma}{(1 -\gamma)} + \1(\gamma > \gamma_{k}) \cdot \frac{\gamma_{k}^2(1-\gamma)}{\gamma}
\Big).
\end{align*}
\end{lemma}
\begin{proof}[Proof of Lemma \ref{Lemma:impurity_decrease}]
Since $k \in U(t)$, we know that $k$ is not in $\parent (t)$. That means any of $t$'s parents do not split on $k$. In other words,  $R_t$ does not have any constraints for feature $k$, i.e., $c_{low, k} = 0$ and $c_{high, k} = 1$.
Thus, we know that
\begin{align}
    \mu(R_{t,l}(\gamma; k)) =\mu(R_t) \cdot \gamma \label{Eq:left_region}
\end{align} 
and 
\begin{align}
\mu(R_{t,r}( \gamma; k)) =\mu(R_t) \cdot (1 - \gamma).\label{Eq:right_region}    
\end{align}
Recall that $\Delta_I$ in \eqref{eq:posMdi} has its equivalent formula \eqref{Eq:equivalent_impurity_decrease}:
\begin{align*} 
\Delta_I(R_{t,l}(\gamma; k), R_{t,r}(\gamma; k)) = \frac{\mu(R_{t,l}(\gamma; k))\mu(R_{t,r}(\gamma; k))}{\mu(R_t(\gamma; k))}\Big[\mathbb E (Y|{X\in R_{t,l}(\gamma; k)}) - \mathbb E (Y|{X\in R_{t,r}(\gamma; k)})\Big]^2
\end{align*}
where the conditional expectations are
\begin{align}\label{Eq:conditional_expect_l}
    \E(Y|X\in R_{t,l}(\gamma; k)) = \sum_{j'=1}^J \beta_{j'} P\left(\forall \ell \in S_{j'},\; X_\ell\leq \gamma_{\ell} \Big| X\in R_{t,l}( \gamma; k)\right),
\end{align}
and
\begin{align}\label{Eq:conditional_expect_r}
    \E(Y|X\in R_{t,r}(\gamma; k)) = \sum_{j'=1}^J \beta_{j'} P\left(\forall \ell \in S_{j'},\; X_\ell\leq \gamma_{\ell} \Big| X\in R_{t,r}( \gamma; k)\right).
\end{align}

Now we will analyze \eqref{Eq:conditional_expect_l} and \eqref{Eq:conditional_expect_r}. To ease the notations, we define the following three events:
\begin{align}
    A_{j'} = & \{X_\ell\leq \gamma_{\ell}, \, \forall \ell \in S_{j'}\},\\
    B = & \{X\in R_t\},\\
    C_{k} = & \{X_k \leq \gamma\}.
\end{align}
Then \eqref{Eq:conditional_expect_l} becomes $\sum_{j'=1}^J \beta_{j'}P(A_{j'}|BC_k)$. 
Because $R_t$ has no constraints on $k$, $B$ does not involve feature $k$. When $j'\neq j$ (namely, $k\not\in S_{j'}$), $A_{j'}$ also does not involve feature $k$. Thus, $C_k$ is independent of $(A_{j'}, B)$, which implies $P(A_{j'}|BC_k) = \frac{P(A_{j'}BC_k)}{P(BC_k)} = \frac{P(A_{j'}B)P(C_k)}{P(B)P(C_k)} = P(A_{j'}|B)$. Similarly this holds for \eqref{Eq:conditional_expect_r}. Therefore, when $j'\neq j$:
\begin{align*}
P\left(\forall \ell \in S_{j'},\; X_\ell\leq \gamma_{\ell} \Big| X\in R_{t,l}( \gamma; k)\right) & = P\left(\forall \ell \in S_{j'},\; X_\ell\leq \gamma_{\ell} \Big| X\in R_{t,r}( \gamma; k)\right).
\end{align*}
When $j' = j$, 
\begin{align*}
& P\left(\forall \ell \in S_{j},\; X_\ell\leq \gamma_{\ell} \Big| X\in R_{t,l}( \gamma; k)\right) - P\left(\forall \ell \in S_{j},\; X_\ell\leq \gamma_{\ell} \Big| X\in R_{t,r}( \gamma; k)\right)\\
(X_k\text{ is ind. of }X_\ell \text{ for } \ell\neq k)=&P(\forall \; \ell\in S_j/\{k\},\; X_\ell \leq \gamma_{\ell}| X\in R_t)\cdot\\
& \qquad\qquad\qquad \Big(P(X_k\leq \gamma_{k}|X_k\leq \gamma) - P(X_k \leq \gamma_{k}|X_k> \gamma)\Big)\\
=&P(\forall \; \ell\in S_j/\{k\},\; X_\ell \leq \gamma_{\ell}| X\in R_t)\cdot \\
&\qquad\qquad
\Big(\1(\gamma \leq \gamma_{k})\cdot \frac{1 - \gamma_{k}}{1 -\gamma} + \1(\gamma > \gamma_{k}) \cdot \frac{\gamma_{k}}{\gamma}
\Big).
\end{align*}
Therefore, \eqref{Eq:equivalent_impurity_decrease} becomes: 
\begin{align}
& \frac{\mu(R_{t,l}(\gamma;k))\mu(R_{t,r}(\gamma;k))}{\mu(R_t)}\Big(\mathbb E(Y|X\in R_{t,l}(\gamma;k)) - \mathbb E(Y|X\in R_{t,r}(\gamma;k))\Big)^2\nonumber\\
&=\mu(R_t)\gamma(1-\gamma)\cdot \beta_{j}^2P(\forall \; \ell\in S_j/\{k\},\; X_\ell \leq \gamma_{\ell}| X\in R_t)^2\cdot \nonumber\\
&\qquad\qquad\qquad\qquad
\Big(\1(\gamma \leq \gamma_{k})\cdot \frac{(1 - \gamma_{k})^2}{(1 -\gamma)^2} + \1(\gamma > \gamma_{k}) \cdot \frac{\gamma_{k}^2}{\gamma^{2}}
\Big)\nonumber\\
&=\mu(R_t)\cdot \beta_{j}^2P(\forall \; \ell\in S_j/\{k\},\; X_\ell \leq \gamma_{\ell}| X\in R_t)^2\cdot \nonumber\\
&\qquad\qquad\qquad\qquad
\Big(\1(\gamma \leq \gamma_{k})\cdot \frac{(1 - \gamma_{k})^2\gamma}{(1 -\gamma)} + \1(\gamma > \gamma_{k}) \cdot \frac{\gamma_{k}^2(1-\gamma)}{\gamma}
\Big).\nonumber
\end{align}
That completes the proof.
\end{proof}
\begin{lemma}\label{Lemma:sufficient_overlap}
For $T\in \mathcal T_0$, $t\in T$, if there exists $j\in [J]$ and $k\in S_j$ such that $k\in \mathring U(t)$, then 
\[
P(\forall \; \ell\in S_j/\{k\},\; X_\ell \leq \gamma_{\ell}| X\in R_t) \geq C_\gamma^{s_j - 1}.
\]
\end{lemma}
\begin{proof}[Proof of Lemma \ref{Lemma:sufficient_overlap}]
Because $k \in S_j$ and $k\in \mathring U(t)$, we know that $S_j^+\cap \dot \parent^\pm(t) = \emptyset$. That means node $t$ is not at the right branch of any node that splits on features in $S_j$. Thus,
\begin{align}\label{Eq:lower_bound_is_zero}
    c_{low,\ell} = & 0 \text{ when }\ell\in S_j.
\end{align}
Also, $c_{high,k} = 1$ and $c_{low,k} = 0$ because $k\in \mathring U(t)$. 
Then, $P(\forall \; \ell\in S_j/\{k\},\; X_\ell \leq \gamma_{\ell}| X\in R_t)$ is
\begin{align}
    & \frac{P (\forall \ell \in S_{j}/\{k\}\; X_\ell\leq \gamma_{\ell}, X\in R_t)}{\mu(R_t)} \nonumber\\
    (\text{Due to \eqref{Eq:lower_bound_is_zero}})= & \frac{\prod_{\ell\in [p]/S_{j}} (c_{high,\ell} - c_{low,\ell})\prod_{\ell\in S_{j}/\{k\}}\min(c_{high,\ell}, \gamma_{\ell})}{\mu(R_t)}\nonumber\\
    \geq & \frac{\prod_{\ell\in [p]/S_{j}} (c_{high,\ell} - c_{low,\ell})\prod_{\ell\in S_{j}/\{k\}}c_{high,\ell}\cdot \gamma_{\ell}}{\mu(R_t)}\nonumber\\
    = & \frac{\mu(R_t) \cdot \prod_{\ell\in S_{j}/\{k\}} \gamma_{\ell}}{\mu(R_t)}\nonumber\\
    \geq & C_\gamma^{s_{j} - 1}. \nonumber
\end{align}
That completes the proof.
\end{proof}
\begin{lemma}\label{Lemma:5} Suppose that constraint C4
from the main text holds. Then, for any fixed $\epsilon > 0$ it holds true that
\begin{align}\label{Eq:desirable_feature_lower_bound}
    P\left(\inf_{T\in \mathcal T_0}\min_{\begin{array}{c} t\in T, \mu(R_t) \geq \epsilon,\\\mathring U(t)\neq \emptyset\end{array}}\min_{k\in \mathring U(t)}\sup_{\gamma\in [C_\gamma, 1 - C_\gamma]} \Delta_I^n(R_{t,l}(\gamma; k), R_{t,r}(\gamma; k)) > \frac{\epsilon}{4}C_\beta^2C_\gamma^{2\max_j s_j - 1}\right) \rightarrow 1.
\end{align}
\end{lemma}

\begin{proof}
First of all, we know from Proposition \ref{Prop:uniform_convergence} that $\sup_{R_t\in \mathscr{R}}|\Delta_I^n(R_t) - \Delta_I(R_t)|\pto 0$. Thus, in order to prove \eqref{Eq:desirable_feature_lower_bound}, we only need to show that 
\begin{align}\label{Eq:to_prove}
    \inf_{T\in \mathcal T_0}\min_{\begin{array}{c} t\in T, \mu(R_t) \geq \epsilon,\\ \mathring U(t)\neq \emptyset\end{array}}\min_{k\in \mathring U(t)} \Delta_I(R_{t,l}(\gamma_{k}; k), R_{t,r}(\gamma_{k}; k)) > \frac{\epsilon}{2}C_\beta^2C_\gamma^{2\max_j s_j - 1}.
\end{align}
Recall that $\gamma_{k}$ is the ground-truth threshold of feature $k$ in the interaction. Here we can drop $\max_{\gamma\in [C_\gamma, 1 - C_\gamma]}$ and use $\gamma_k$ because that results in a lower bound of the previous equation.
Based on Lemma \ref{Lemma:impurity_decrease}, we know that 
\begin{align*}
& \Delta_I(R_{t,l}(\gamma_k;k), R_{t,r}(\gamma_k;k)) \\
=& \mu(R_t)\cdot \beta_{j}^2P(\forall \; \ell\in S_j/\{k\},\; X_\ell \leq \gamma_{\ell}| X\in R_t)^2\cdot (1 - \gamma_{k})\gamma_k\\
\geq & \frac{1}{2}C_\gamma C_\beta^2 \epsilon\cdot P(\forall \; \ell\in S_j/\{k\},\; X_\ell \leq \gamma_{\ell}| X\in R_t)^2.
\end{align*}
The second inequality is due to $\mu(R_t)\geq \epsilon$, $\gamma_{k}(1-\gamma_{k})\geq C_\gamma(1-C_\gamma) \geq \frac{1}{2} C_\gamma$ and $\beta_j \geq C_\beta$.
Then using Lemma \ref{Lemma:sufficient_overlap} leads to the conclusion.
\end{proof}
For a node $t$, denote $\gamma^*_{t,k} = \mathrm{argmax}_{\gamma\in [C_\gamma, 1-C_\gamma]} \Delta_I^n(R_{t,l}(\gamma;k), R_{t,r}( \gamma;k))$.
\begin{lemma}\label{lemma:threshold} Suppose that constraint C4
from the main text holds true, then we have
\[\sup_{T\in \mathcal T_0}\max_{\begin{array}{c} t\in T, \mu(R_t)\geq \epsilon,\\ \mathring U(t)\neq \emptyset\end{array}} \max_{k\in \mathring U(t)} |\gamma^*_{t,k} - \gamma_{k}|\pto 0.
\]
\end{lemma}
\begin{proof} 
To simplify the notation in the proof, let us denote
\begin{align*}
    a_n &= \Delta_I^n(R_{t,l}(\gamma^*_{t,k};k), R_{t,r}(\gamma^*_{t,k};k)),\\
    a &= \Delta_I(R_{t,l}(\gamma^*_{t,k};k), R_{t,r}(\gamma^*_{t,k};k))\\
    b_n &= \Delta_I^n(R_{t,l}(\gamma_{k};k), R_{t,r}(\gamma_{k};k)),\\
    b &= \Delta_I(R_{t,l}(\gamma_{k};k), R_{t,r}(\gamma_{k};k)).
\end{align*}
Using Proposition \ref{Prop:uniform_convergence}, we have 
\begin{align}
\sup_{T\in \mathcal T_0}\max_{\begin{array}{c} t\in T, \mu(R_t)\geq \epsilon,\\ \mathring U(t)\neq \emptyset\end{array}} \max_{k\in \mathring U(t)} \Big|a_n - a\Big|\pto 0.
\end{align}
By Lemma \ref{Lemma:5} (see \eqref{Eq:to_prove}), we know the second term is bounded uniformly above zero:
\[
\inf_{T\in \mathcal T_0}\min_{\begin{array}{c}t\in T, \mu(R_t)\geq \epsilon, \\ \mathring U(t)\neq \emptyset\end{array}} \min_{k\in \mathring U(t)} a \geq \frac{\epsilon}{2}C_\beta^2C_\gamma^{2\max_j s_j - 1}.
\]
Thus, the ratio converges to 1 in probability:
\begin{align}
\sup_{T\in \mathcal T_0}\max_{\begin{array}{c}t\in T, \mu(R_t)\geq \epsilon, \\ \mathring U(t)\neq \emptyset\end{array}} \max_{k\in \mathring U(t)} \left|\frac{a_n}{a} - 1\right|\pto 0.
\end{align}
Similarly, this  applies to $b_n$ and $b$, i.e.,
\begin{align}
\sup_{T\in \mathcal T_0}\max_{\begin{array}{c}t\in T, \mu(R_t)\geq \epsilon, \\ \mathring U(t)\neq \emptyset\end{array}} \max_{k\in \mathring U(t)} \left|\frac{b_n}{b} - 1\right|\pto 0.
\end{align}
 So by the continuous mapping theorem, we know that 
\[
\sup_{T\in \mathcal T_0}\max_{\begin{array}{c}t\in T, \mu(R_t)\geq \epsilon, \\ \mathring U(t)\neq \emptyset\end{array}} \max_{k\in \mathring U(t)}  \left| \frac{b_n}{a_n}\frac{a}{b} - 1\right| \pto 0.
\]
Because $\gamma^*_{t,k}$ maximizes $\Delta_I^n$ and $\gamma_{k}$ maximizes $\Delta_I$,  $a_n\geq b_n$ and $a \leq b$. Thus $\frac{b_n}{a_n}\frac{a}{b} \leq \frac{a}{b} \leq 1$. Therefore, we know that
\[
\sup_{T\in \mathcal T_0}\max_{\begin{array}{c}t\in T, \mu(R_t)\geq \epsilon, \\ \mathring U(t)\neq \emptyset\end{array}} \max_{k\in \mathring U(t)} 1 - \frac{a}{b} \pto 0.
\]
By Lemma \ref{Lemma:impurity_decrease}, we know that
\begin{align*}
a = &\mu(R_t)\cdot \beta_{j}^2P(\forall \; \ell\in S_j/\{k\},\; X_\ell \leq \gamma_{\ell}| X\in R_t)^2\cdot 
\Big(\1(\gamma^*_{t,k} \leq \gamma_{k})\cdot \frac{(1 - \gamma_{k})^2\gamma^*_{t,k}}{(1 -\gamma^*_{t,k})} + \1(\gamma^*_{t,k} > \gamma_{k}) \cdot \frac{\gamma_{k}^2(1-\gamma^*_{t,k})}{\gamma^*_{t,k}}
\Big),\\
    b = & \mu(R_t)\cdot \beta_{j}^2P(\forall \; \ell\in S_j/\{k\},\; X_\ell \leq \gamma_{\ell}| X\in R_t)^2\cdot 
 \gamma_{k}(1-\gamma_{k}).
\end{align*}
Thus the ratio is 
\begin{align*}
\frac{a}{b}& =\1(\gamma^*_{t,k} \leq \gamma_{k})\cdot \frac{(1 - \gamma_{k})\gamma^*_{t,k}}{\gamma_{k}(1 -\gamma^*_{t,k})} + \1(\gamma^*_{t,k} > \gamma_{k}) \cdot \frac{\gamma_{k}(1-\gamma^*_{t,k})}{(1-\gamma_{k})\gamma^*_{t,k}}.
\end{align*}
When $\gamma^*_{t,k}\leq \gamma_{k}$, 
\begin{align*}
    1 - \frac{a}{b} & = 1 - \frac{(1 - \gamma_{k})\gamma^*_{t,k}}{\gamma_{k}(1 -\gamma^*_{t,k})} \\
    & =  \frac{\gamma_{k} - \gamma^*_{t,k}}{\gamma_{k}(1 -\gamma^*_{t,k})} \geq \gamma_{k} - \gamma^*_{t,k}.
\end{align*}
Similarly, when $\gamma^*_{t,k}\geq \gamma_{k}$, then $1 - \frac{a}{b}\geq \gamma^*_{t,k} - \gamma_{k}$. 
Thus, $1 - a / b \geq |\gamma_{k} - \gamma^*_{t,k}|\geq 0$. Thus, by the Squeeze theorem, we complete the proof.
\end{proof}

\begin{lemma}\label{lemma:to_U} 
Suppose that 
constraint C4
from the main text holds. Then the following statements are true: 
\begin{enumerate}
    \item[i)] For any fixed $\epsilon, \delta > 0$,
\begin{align*}
    P\Bigg(&\inf_{T\in \mathcal T_1(\mathcal D)}\min_{\begin{array}{c} t\in T, \mu(R_t)\geq \epsilon\\ U(t)\neq \emptyset\end{array}}\min_{j\in[J]}\min_{k\in S_j\cap U(t)}\\
    &P(\forall \; \ell\in S_j/\{k\},\; X_\ell \leq \gamma_{\ell}| X\in R_t; \mathcal D ) - C_\gamma^{s_j - 1}
    \geq -\delta\Bigg)\to 1.
\end{align*}
\item[ii)] For any fixed $\epsilon > 0$, 
\begin{align*}
    P\Bigg(&\inf_{T\in \mathcal T_1(\mathcal D)}\min_{\begin{array}{c} t\in T, \mu(R_t) \geq \epsilon,\\ U(t)\neq \emptyset\end{array}}\min_{k\in U(t)}\sup_{\gamma\in [C_\gamma, 1 - C_\gamma]}\\
    &\Delta_I^n(R_{t,l}(\gamma; k), R_{t,r}(\gamma; k)) > \frac{\epsilon}{4}C_\beta^2C_\gamma^{2\max_j s_j - 1}\Bigg) \rightarrow 1.
\end{align*}
\item[iii)]
\[\sup_{T\in \mathcal T_1(\mathcal D)}\max_{\begin{array}{c} t\in T, \mu(R_t)\geq \epsilon,\\ U(t)\neq \emptyset\end{array}} \max_{k\in U(t)} |\gamma^*_{t,k} - \gamma_{k}|\pto 0.
\]
\end{enumerate}

\end{lemma}
\begin{proof}
We use math induction to show that the above statements hold for any $L\geq 0$:
\begin{enumerate}
    \item[i)] For any fixed $\epsilon, \delta > 0$,
    \begin{align*}\small
       P\Bigg(&\inf_{T\in \mathcal T_1(\mathcal D)}\min_{\begin{array}{c} t\in T, \mu(R_t)\geq \epsilon, U(t)\neq \emptyset, \\\sum_{j=1}^J|\dot \parent^\pm(t) \cap S_j^+|\leq L\end{array}}\min_{j\in[J]}\min_{k\in S_j\cap U(t)} \\ 
       & P(\forall \; \ell\in S_j/\{k\},\; X_\ell \leq \gamma_{\ell}| X\in R_t; \mathcal D ) - C_\gamma^{s_j - 1} \geq -\delta\Bigg)\to 1. 
    \end{align*}
\item[ii)] For any fixed $\epsilon > 0$, 
\begin{align*}
    P\Bigg(&\inf_{T\in \mathcal T_1(\mathcal D)}\min_{\begin{array}{c} t\in T, \mu(R_t) \geq \epsilon, U(t)\neq \emptyset, \\\sum_{j=1}^J|\dot \parent^\pm(t) \cap S_j^+|\leq L\end{array}}\min_{k\in U(t)}\sup_{\gamma\in [C_\gamma, 1 - C_\gamma]}\\
    &\Delta_I^n(R_{t,l}(\gamma; k), R_{t,r}(\gamma; k)) > \frac{\epsilon}{4}C_\beta^2C_\gamma^{2\max_j s_j - 1}\Bigg) \rightarrow 1.
\end{align*}
\item[iii)]
\[\sup_{T\in \mathcal T_1(\mathcal D)}\max_{\begin{array}{c} t\in T, \mu(R_t)\geq \epsilon, U(t)\neq \emptyset, \\\sum_{j=1}^J|\dot \parent^\pm(t) \cap S_j^+|\leq L\end{array}} \max_{k\in U(t)} |\gamma^*_{t,k} - \gamma_{k}|\pto 0.
\]
\end{enumerate}
If those statements are true, then our proof is complete because for any node $t$, $\sum_{j=1}^J |\dot \parent^\pm(t) \cap S_j^+| \leq \sum_{j} s_j$, which is a constant.

When $L = 0$, $U(t)\neq \emptyset$ and $\sum_{j} |\dot \parent^\pm(t) \cap S_j^+| = 0$ implies that $U(t) = \cup_{j=1}^JS_j / F(t) = \mathring U(t)\neq \emptyset$. Then the statement holds because of Lemmas \ref{Lemma:sufficient_overlap}, \ref{Lemma:5}, and \ref{lemma:threshold}. 

Suppose the statement holds for $L=L_0$, and let us consider the case $L = L_0+1$:

i): For $k \in S_j\cap U(t)$, we know that $S_j^+\cap \parent^\pm(t) = \emptyset$.
Now consider $S_j^+\cap \dot \parent^\pm(t)$: if it is also empty, then $k\in \mathring U(t)$ and i) holds because of Lemma \ref{Lemma:sufficient_overlap}. 
Let's consider the case when $S_j^+\cap \dot \parent^\pm(t)\neq \emptyset$.
For any $\ell\in S_j^+\cap \dot \parent^\pm(t)$, some parent nodes of $t$ are split on feature $\ell$ and node $t$ is at the left branch of the first parent node that is split on $\ell$. 
In other words, this is the scenario where $(\ell, -1)$ first appears in the path and then $(\ell, +1)$ appears later. Denote that first parent node that is split on $\ell$ to be $t_{parent,\ell}$. 
Since none of $t_{parent,\ell}$'s parent nodes are split on $\ell$, $\ell\in S_j^+\cap \dot \parent^\pm(t)$ but not in $S_j^+\cap \dot \parent^\pm(t_{parent,\ell})$. 
Since $\dot \parent^\pm(t_{parent,\ell})$ is a subset of $\dot \parent^\pm(t)$, we know that $\sum_{j=1}^J |S_j^+\cap \dot \parent^\pm(t_{parent,\ell})| \leq L_0$.
Also, because $S_j^+\cap \dot \parent^\pm(t_{parent,\ell})=\emptyset$ and $\ell\not \in \dot F(t_{parent,\ell})$, we know that
$\ell\in U(t_{parent,\ell})$. 
Then by the induction condition iii), we know that $\gamma^*_{t_{parent,\ell}, \ell} \pto \gamma_\ell$. 
Because $t$ is at the left branch of $t_{parent, \ell}$, the upper bound in $R_t$ for feature $\ell$, i.e., $c_{high,\ell}$, is smaller or equal to $\gamma^*_{t_{parent,\ell},\ell}$. 
In other words, for any fixed $\delta > 0$, we know that
\[
P\left(\sup_{T\in \mathcal T_1(\mathcal D)}\max_{\begin{array}{c} t\in T, \mu(R_t)\geq \epsilon, U(t)\neq \emptyset, \\\sum_{j=1}^J|\dot \parent^\pm(t) \cap S_j^+|\leq L_0+1\end{array}} \max_{\begin{array}{c}j\in[J]\\S_j^+\cap \parent^\pm(t)=\emptyset\end{array}}\max_{(\ell,+1)\in S_j^+\cap \dot \parent^\pm(t)} c_{high,\ell} - \gamma_{\ell} > \delta\right)\pto 0.
\]
For any $l$ such that $\ell\in S_j$ but $(\ell,+1)\not\in S_j^+ \cap \dot \parent^\pm(t)$, we have that $c_{low, \ell} = 0$. 
Note that $c_{high,k} = 1$ and $c_{low,k} = 0$ because $k\in U(t)$. 
Then, $P(\forall \; \ell\in S_j/\{k\},\; X_\ell \leq \gamma_{\ell}| X\in R_t;\mathcal D)$ is
\begin{align}
    & \frac{P (\forall \ell \in S_{j}/\{k\}\; X_\ell\leq \gamma_{\ell}, X\in R_t; \mathcal D)}{\mu(R_t)} \nonumber\\
    = & \frac{\prod_{\ell\in [p]/S_{j}} (c_{high,\ell} - c_{low,\ell})\prod_{\ell\in S_{j}/\{k\}}\max(\min(c_{high,\ell}, \gamma_{\ell})-c_{low,\ell}, 0) }{\mu(R_t)}\nonumber\\
    = & \frac{\prod_{\ell\in [p]/S_{j}} (c_{high,\ell} - c_{low,\ell})\prod_{(\ell,+1)\in S_{j}^+\cap \dot \parent^\pm(t)}(c_{high,\ell}-c_{low,\ell} + o_p(1)) \prod_{\ell\in S_j/\{k\}, (\ell,+1)\not\in S_{j}^+\cap \dot \parent^\pm(t)}\min(c_{high,\ell},\gamma_\ell)}{\mu(R_t)}\nonumber\\
    \geq & \frac{\prod_{\ell\in [p]/S_{j}} (c_{high,\ell} - c_{low,\ell})\prod_{(\ell,+1)\in S_{j}^+\cap \dot \parent^\pm(t)}(c_{high,\ell}-c_{low,\ell}) \prod_{\ell\in S_j/\{k\}, (\ell,+1)\not\in S_{j}^+\cap \dot \parent^\pm(t)}c_{high,\ell}\cdot\gamma_\ell}{\mu(R_t)} + o_p(1)\nonumber\\
    \geq & \frac{\mu(R_t) \cdot \prod_{\ell\in S_j/\{k\}, (\ell,+1)\not\in S_{j}^+\cap \dot \parent^\pm(t)} \gamma_{\ell}}{\mu(R_t)} + o_p(1)\nonumber\\
    \geq & C_\gamma^{s_{j} - 1}  + o_p(1), \nonumber
\end{align}
where the first equality follows from \eqref{Eq:lower_bound_is_zero}.
That completes the proof for i).

ii): Given i), ii) follows analog as in the proof of Lemma \ref{Lemma:5}.

iii) Given ii), iii) follows analog as in the proof of Lemma \ref{lemma:threshold}.

Thus, we have finished the math induction and proved the statements.
\end{proof}

\begin{lemma}\label{lemma:noisy_feature}
For any tree $T\in \mathcal T_1$ and any node $t\in T$, the noisy features correspond to a nearly zero impurity decrease, i.e. 
\begin{align}
\sup_{T\in \mathcal T_1}\max_{t\in T}\max_{k\in [p]/U(t)}\sup_{\gamma\in [0, 1]} \Delta_I^n(R_{t,l}(\gamma;k), R_{t,r}(\gamma;k)) \pto 0.
\end{align}
\end{lemma}
\begin{proof}
By Proposition \ref{Prop:uniform_convergence}, we only need to show that 
\begin{align}
\sup_{T\in \mathcal T_1}\max_{t\in T}\max_{k\in [p]/U(t)}\sup_{\gamma\in [0, 1]} \Delta_I(R_{t,l}(\gamma;k), R_{t,r}(\gamma;k)) \pto 0.
\end{align}
For $k\in [p]/U(t)$, either $k\in [p] / \bigcup_{j=1}^J S_j$ or $k\in \bigcup_{j=1}^J S_j / U(t)$. We will analyze these two cases separately:

\vspace{0.1cm}

First, assume that $k\in [p] / \bigcup_{j=1}^J S_j$. For any $j'\in [J]$, it follows that $k$ is not contained in $S_{j'}$. 
Because different features are independent, $X_k$ is independent from $X\in \{X|\forall \; \ell\in S_{j'},\; X_\ell \leq \gamma_{\ell}\}$. 
Therefore, for any $j'\in [J]$, we have 
\[P(\forall \; \ell\in S_{j'},\; X_\ell \leq \gamma_{\ell}| X\in R_{t,l}(\gamma;k)) = P(\forall \; \ell\in S_{j'},\; X_\ell \leq \gamma_{\ell}| X\in R_{t,r}(\gamma;k)).\]
That implies $\Delta_I(R_{t,l}(\gamma;k), R_{t,r}(\gamma;k)) = 0$. 

\vspace{0.1cm}

Second, assume that there exists $j$ such that $k\in S_j/U(t)$. 
For $j'\neq j$, by a similar deduction as before, we know that 
\[P(\forall \; \ell\in S_{j'},\; X_\ell \leq \gamma_{\ell}| X\in R_{t,l}(\gamma;k)) = P(\forall \; \ell\in S_{j'},\; X_\ell \leq \gamma_{\ell}| X\in R_{t,r}(\gamma;k)).\]
The impurity decrease $\Delta_I(R_{t,l}(\gamma;k), R_{t,r}(\gamma;k))$ becomes
\begin{align}\label{Eq:bridge}
\frac{\mu(R_{t,l}(\gamma;k))\mu(R_{t,r}(\gamma;k))}{\mu(R_t)} \beta_j^2\Big(P(\forall \; \ell\in S_{j},\; X_\ell \leq \gamma_{\ell}| X\in R_{t,l}(\gamma;k)) - P(\forall \; \ell\in S_{j},\; X_\ell \leq \gamma_{\ell}| X\in R_{t,r}(\gamma;k)) \Big)^2.
\end{align}
Again, we consider two cases: Because $k\not \in U(t)$, either $(k, -1)\in \parent^\pm(t)$ or $S_j^+\cap \parent^\pm(t) \neq \emptyset$. 

\vspace{0.1cm}

i) If $S_j^+\cap \parent^\pm(t)\neq \emptyset$, suppose $(k', +1)$ is the first positive signed feature in $S_j^+$ that enters $\parent^\pm(t)$. 
That means we can find a parent of $t$, denoted as $t_{parent}$, that splits on feature $k'$ and none of $t_{parent}$'s parent splits on $k'$. 
That implies $k'\not\in \parent(t_{parent})$ and $S_j^+\cap \parent^\pm(t_{parent}) = \emptyset$, in other words, $k'\in U(t_{parent})$. 
Recall that $\gamma_{t_{parent},k'}^*$ denotes the threshold at node $t_{parent}$. 
By Lemma \ref{lemma:to_U}, we know that the threshold $\gamma_{t_{parent},k'}^* \pto \gamma_{k'}$. 
Since $t$ is on the right branch of the node $t_{parent}$, we have that $c_{low, k'}(t)\geq \gamma_{t_{parent},k'}^*$. 
Thus, \[\mu(\{X|\forall \; \ell\in S_{j},\; X_\ell \leq \gamma_{\ell}\}\cap R_{t}) \pto 0.\] 
Since \eqref{Eq:bridge} is bounded by 
\begin{align*}
    2C_\beta^2 \frac{\mu(R_{t,l}(\gamma;k))\mu(R_{t,r}(\gamma;k))}{\mu(R_t)} \Big[ P(\forall \; \ell\in S_{j},\; X_\ell \leq \gamma_{\ell}| X\in R_{t,l}(\gamma;k)) + P(\forall \; \ell\in S_{j},\; X_\ell \leq \gamma_{\ell}| X\in R_{t,r}(\gamma;k))\Big]\\
    \leq 2C_\beta^2 P(\forall \; \ell\in S_{j},\; X_\ell \leq \gamma_{\ell}, X\in R_t),
\end{align*}
we know that \eqref{Eq:bridge} converges to zero in probability.

\vspace{0.1cm}

ii) If $S_j^+\cap \parent^\pm(t) = \emptyset$ but $(k, -1)\in \parent^\pm(t)$, it means there exists a parent of $t$, denoted $t_{parent}$, such that feature $k$ is used to split that node and none of its parents splits on $k$, in other words, $k\in U(t_{parent})$. 
By Lemma \ref{lemma:to_U}, we know that the corresponding threshold $\gamma_{t_{parent}, k}^* \pto \gamma_{k}$. 
Since $S_j^+\cap \parent^\pm(t)=\emptyset$, it follows that $t$ is on the left branch of $t_{parent}$. 
Thus, we have that $c_{high, k}(t)\leq \gamma_{t_{parent},k}^*$. 
For any fixed $\epsilon > 0$, if $\mu(R_{t,l}(\gamma;k)) > \epsilon$ and $\mu(R_{t,r}(\gamma;k)) > \epsilon$, then 
\[P(\forall \; \ell\in S_{j},\; X_\ell \leq \gamma_{\ell,j}| X\in R_{t,l}(\gamma;k)) - P(\forall \; \ell\in S_{j},\; X_\ell \leq \gamma_{\ell,j}| X\in R_{t,r}(\gamma;k)) \pto 0,\] 
which implies that $\Delta_I(R_{t,l}(\gamma;k), R_{t,r}(\gamma;k))\pto 0$.
Otherwise, $[\mu(R_{t,l}(\gamma;k)) \leq \epsilon$ or $\mu(R_{t,r}(\gamma;k)) \leq \epsilon$,
and thus, \[\frac{\mu(R_{t,l}(\gamma;k))\mu(R_{t,r}(\gamma;k))}{\mu(R_t)}\leq \epsilon\] and 
\[\Delta_I(R_{t,l}(\gamma;k), R_{t,r}(\gamma;k))\leq 4\epsilon.\] 
Since $\epsilon$ is chosen arbitrarily, this implies $\Delta_I(R_{t,l}(\gamma;k), R_{t,r}(\gamma;k))\pto 0$.

\vspace{0.1cm}

Combining a) and b), we complete the proof.

\end{proof}

With the help of the previous lemmas, we have the following proposition:
\begin{proposition} \label{prop:noisyCase}
Suppose $t_{\leaf}$ is a leaf of $\cP$ from a random tree $T\in \mathcal T_2$. Suppose that 
constraint C4 and assumptions \ref{A:increasing_depth}-\ref{A:no_bootstrap} from the main text hold true. 
For any fixed constant $\epsilon > 0$,
the following holds true:
\begin{enumerate}
\item[i)]
\[
P\left(\max_{t\in T} \max_{k\in [p]/U(t)} \Delta_I^n (R_{t,l}(\gamma^*_{t,k};k), R_{t,r}(\gamma^*_{t,k};k)) < \frac{\epsilon}{4}C_\beta^2C_\gamma^{2\max_j s_j - 1}
\right) \to
1.
\]
\item[ii)] 
\[
P\left(\text{$U(t_{\leaf}) = \emptyset$}\Big| \cD\right) \pto 1.
\]
\item[iii)]
\[
P\left(\min_{t\in \cp(t_{\leaf})} \min_{k\in U(t)} \Delta_I^n (R_{t,l}(\gamma^*_{t,k};k), R_{t,r}( \gamma^*_{t,k};k)) \geq \frac{\epsilon}{4}C_\beta^2C_\gamma^{2\max_j s_j - 1} \Big| \cD
\right) \geq
1 - \epsilon^{\tilde{C}} - \eta_n(\cD, \epsilon),
\]
with constant $\tilde{C} = C_m^{2 s} / \log(1/C_\gamma)$ and $\eta_n(\cD, \epsilon)\pto 0$.
\end{enumerate}
\end{proposition}
\begin{proof}
i) By Lemma \ref{lemma:noisy_feature}, we know with probability approaching 1,
\begin{align}\label{Eq:upper}
\max_{t\in T}\max_{k\in [p]/U(t)}\sup_{\gamma\in [0, 1]} \Delta_I^n(R_{t,l}(\gamma;k), R_{t,r}(\gamma;k)) \leq \frac{\epsilon}{4}C_\beta^2C_\gamma^{2\max_j s_j - 1}.
\end{align}

ii): 
For any fixed $\epsilon > 0$, by Lemma \ref{Lemma:5} and Lemma \ref{lemma:noisy_feature}, the following event $A_\epsilon$ happens with probability approaching 1,
\begin{align}\label{Eq:lower}
\begin{aligned}
    &A_\epsilon =\\
    &\bigcap_{T\in\mathcal T_1}\Bigg\{\min_{t\in T, \mu(R_t) \geq \epsilon,U(t)\neq \emptyset}\min_{k\in U(t)}\sup_{\gamma\in [C_\gamma, 1 - C_\gamma]} \Delta_I^n(R_{t,l}(\gamma;k)\}, R_{t,r}(\gamma;k))\\
    &> \max_{t\in T}\max_{k\in [p]/U(t)}\sup_{\gamma\in [0, 1]} \Delta_I^n(R_{t,l}(\gamma;k), R_{t,r}(\gamma;k))\Bigg\},
\end{aligned}
\end{align}
which implies that for any node with volume at least $\epsilon$ any desirable features has higher impurity decrease than any non-desirable feature.
For a random path $\cP$, denote its leaf node $t_{\leaf}$ and the depth of the path is $D$. Then for $d\in [D]$, denote $t_d$ to be the $d$-th node on the path $\cP(t_{\leaf})$. Recall that $S = \cup_{j = 1}^J S_j$ denotes the set of all signal features and $s = |S|$ their total number. 
Based on \eqref{Eq:lower}, if at any node $t$, its candidate feature set $\Mt(t)$ contains all the signal features $S$, then it will split on a signal feature as long as $U(t_{\leaf}) \neq \emptyset$. If there are more than $s$ nodes along the path that has volume larger than $\epsilon$ and their candidate feature set contains $S$, then the desirable features must have been exhausted at the leaf node, i.e., 
\begin{align}
\left\{ \Big|\{d \in [D] : S \subset \Mt(t_d) \text{ and } \mu(R_d) \geq \epsilon\}\Big| \geq s, A_\epsilon \right\} \subset \{ U(t_{\leaf}) = \emptyset, A_\epsilon\}.
\end{align}
Further, note that, because $\mu(R_{t_d}) \geq C_\gamma \mu(R_{t_{d-1}})\geq\ldots \geq C_\gamma^d$, when $d < \log \epsilon / \log C_\gamma$, it always holds that $\mu(R_{t_{d-1}})\geq \epsilon$ and therefore
\begin{align}\label{eq:Dequation}
\left\{\Big|\{d \in [\log \epsilon / \log C_\gamma] : S \subset \Mt(t_d) \}\Big| \geq s, A_\epsilon, D \geq \log \epsilon / \log C_\gamma \right\}\subset\left\{ U(t_{\leaf}) = \emptyset, A_\epsilon, D \geq \log \epsilon / \log C_\gamma\right\}.
\end{align}
Since for any node $t$, its candidate feature set $\Mt(t)$ has $\mt$ features, we know 
\begin{align*}
    P(S \subset \Mt(t)) = \frac{{p - s \choose \mt - s}}{ {p \choose \mt}} = \frac{\mt\cdot(\mt-1)\cdots(\mt-s+1)}{p\cdot(p-1)\cdots(p-s+1)}\geq \left(\frac{\mt - s + 1}{ p - s + 1} \right)^s \geq [C_m]^s.
\end{align*}
Since $\Mt(t)$ is independent of the path $\cP$, it follows that
\begin{align*}
&P_{(\cP, T)}\left( \Big|\{d \in [\log \epsilon / \log C_\gamma] : S \subset \Mt(t_d) \}\Big| \geq s \Big|D \geq \log \epsilon / \log C_\gamma, \cD\right)\\
\geq \; & P( B(\log \epsilon / \log C_\gamma, [C_m]^s) \geq s) - \1({\cD\in A_\epsilon})\\
\geq \; & 1 - \exp\left( - 2 \log \epsilon / \log C_\gamma \left( [C_m]^s - \frac{s}{\log \epsilon / \log C_\gamma} \right)^2 \right) - \1({\cD\in A_\epsilon})
\end{align*}
where $B(n,p)$ denotes a Binomial distribution with $n$ trails and success probability $p$ and the last inequaility follows from Hoeffding's inequality.
Thus, for any $0 < \epsilon < \exp( (1 - 1/\sqrt{2})[C_m]^s / (s \log (1 / C_\gamma)))$, we have 
\[
\left( [C_m]^s - \frac{s}{\log \epsilon / \log C_\gamma} \right)^2 \geq \frac{1}{2} C_m^{2s}.
\]
Denote
\[\tilde{C} = C_m^{2 s} / \log(1/C_\gamma),\]
we have that for sufficiently large $n$
\begin{align}
P_{(\cP, T)}\left( \Big|\{d \in [\log \epsilon / \log C_\gamma] : S \subset \Mt(t_d) \}\Big| \geq s \Big| D(\cP) \geq \log \epsilon / \log C_\gamma, \cD\right) \geq 1 - \epsilon^{\tilde{C}} - \1({\cD\in A_\epsilon})
\end{align}
and thus it follows from \eqref{eq:Dequation} that
\begin{align}
P_{(\cP, T)}\left(U(t_{\leaf}) = \emptyset \Big| D \geq \log \epsilon / \log C_\gamma, \cD \right) \geq 1 -  \epsilon^{\tilde C}- \1({\cD\in A_\epsilon}).
\end{align}
Because $P(D \geq \log \epsilon / \log C_\gamma) \to 1$, by the Markov inequality, we know the random variable \[P(D \geq \log \epsilon / \log C_\gamma\Big| \cD) \pto 1.\] Thus, we know 
\begin{align}
P_{(\cP, T)}\left(U(t_{\leaf}) = \emptyset \Big| \cD \right) \geq 1 -  \epsilon^{\tilde C} + \eta_n(\cD, \epsilon),
\end{align}
where $\eta_n(\cD, \epsilon)$ is a random variable only depend on $\cD$ and $\eta_n(\cD, \epsilon)\pto 0$. Because that holds for any $\epsilon$, we have \[P_{(\cP, T)}\left(U(t_{\leaf}) = \emptyset \Big| \cD \right)\pto 1.\]

iii) Denote $t_s$ to be the s-th node in a path $\cP(t_{\leaf})$ for $s\geq 1$. 
Based on the proof of ii), let $d$ be an integer that (roughly) equals to $\frac{\log \epsilon}{\log C_\gamma}$. 
Then $\mu(R_{t_d})\geq \epsilon$ and $P(U(t_{d})= \emptyset\Big| \cD) \geq 1 - \epsilon^{\tilde{C}} + \eta_n(\cD , \epsilon)$. 
When $U(t_d) = \emptyset$, it follows that $U(t_s) \neq \emptyset$ implies $s \leq d$ and $\mu(R_{t_s}) \geq \epsilon$. 
Thus, 
\[P(\exists t\in \cP(t_{\leaf}),\text{ such that }U(t)\neq \emptyset \text{ and } \mu(R_t) < \epsilon|\cD) \leq P(U(t_d) \neq \emptyset|\cD) \leq \epsilon^{\tilde{C}} - \eta_n(\cD, \epsilon).\]
Therefore, we have
\begin{align}
&P\left(\min_{t\in \cp(t_{\leaf}), U(t)\neq \emptyset} \min_{k\in U(t)} \Delta_I^n (R_{t,l}(\gamma^*_{t,k};k), R_{t,r}(\gamma^*_{t,k};k)) \geq \frac{\epsilon}{4}C_\beta^2C_\gamma^{2\max_j s_j - 1}
\Big| \cD \right)\nonumber \\
\geq & P\left(\min_{t\in \cp(t_{\leaf}), \mu(R_t)\geq \epsilon, U(t)\neq \emptyset} \min_{k\in U(t)} \Delta_I^n (R_{t,l}(\gamma^*_{t,k};k), R_{t,r}(\gamma^*_{t,k};k)) \geq \frac{\epsilon}{4}C_\beta^2C_\gamma^{2\max_j s_j - 1} \Big| \cD
\right)
 -\epsilon^{\tilde{C}} - \eta_n(\cD, \epsilon),
\end{align}
thus, the proof follows from Lemma \ref{lemma:to_U}.
\end{proof}

\subsubsection{Balanced root feature selection }
Recall the definition of $C_{\mathrm{root}}(\cD)$ in \eqref{eq:invFeat}, which appears in Theorem \ref{theo:SappearNew}.
Recall that for any tree $T$ from RF, there are two different sources of randomness:
first, the randomness of the data $\cD = ((\bx_i, y_i))_{i = 1}^n$, which we denoted as $(\cD)$, and second, the randomness from the candidate feature selection, which we denoted as $(T)$. 
Denote $\Mt(t) \subset [p]$ to be the set of candidate features selected at node $t$ and note that $\Mt(t)$ and the data $\cD$ are independent.

Define the event $A$ to be that, given data $\cD$, the maximum impurity decrease at the split of root node for every signal feature $k \in \cup_j S_j$ is larger than that of any noisy feature $k^\prime \not\in \cup_jS_j$, that is, 
\begin{align}\label{eq:eventA}
    A = \big\{ 
    \min_{k \in \cup_j S_j} \Delta_I^n \left( R_{t_{\mathrm{root}},l}(\gamma^\star_k, k), R_{t_{\mathrm{root}},r}(\gamma^\star_k,k) \right) > \max_{k^\prime \not\in \cup_j S_j} \Delta_I^n\left( R_{t_{\mathrm{root}},l}(\gamma^\star_{k^\prime}, k^\prime), R_{t_{\mathrm{root}},r}(\gamma^\star_{k^\prime},k^\prime) \right) 
    \big\}.
\end{align}
Note that the event only depends on the data randomness $\cD$ (and not on the tree randomness $T$ and the path randomness $\cP$).
Thus, $A$ is independent of $\Mt(t_{\mathrm{root}})$.
Note that it follows from Proposition \ref{prop:noisyCase} that 
\begin{align*}
    P_{\cD}(A) \to 1 \quad \text{ as } n \to \infty.
\end{align*}

\begin{theorem}\label{theo:balancedRootFeature}
Assume that $C_m p \leq \mt \leq (1 - C_m) (p - s + 1) + 1$ for some constant $C_m \in (0,1)$. 
Condition on $\cD = ((\bx_i, y_i))_{i = 1}^n$, for any $k\in \cup_j S_j$, we have that
    \begin{align*}
        P_T\left(t_{\mathrm{root}} \text{ splits on feature }k \middle|\; \cD \right) \geq C_m^{s} - 1_{A^c}
    \end{align*}
and thus, 
\begin{align*}
C_{\mathrm{root}}(\cD) \geq  C_m^{s}  - 1_{A^c}  \pto [C_m]^s \quad \text{ as } n \to \infty.    
\end{align*}
\end{theorem}
\begin{proof}
For any $k\in \cup_j S_j$, define $B_k$ to be the event that only signal feature $k$ is selected in $\Mt(t_{\mathrm{root}})$, that is,
\begin{align*}
        B_k \triangleq \{ \Mt(t_{\mathrm{root}}) \cap \; \cup_j S_j = k \text{ and } \abs{\Mt(t_{\mathrm{root}}) \setminus \cup_j S_j} = \mt - 1\}.
\end{align*}
Note that $B_k$ only depends on $\Mt(t_{\mathrm{root}})$ but not on $\cD$ and 
\begin{align*}
A \cap B_k \subset  \{ t_{\mathrm{root}} \text{ splits on feature }k \}.
\end{align*}
Thus, 
\begin{align*}
     P_T\left(t_{\mathrm{root}} \text{ splits on feature }k \middle| \cD \right) \geq P_T(B_k\cap A|\cD) \geq P_T\left( B_k \middle| \cD \right) - P_T\left( A^c \middle| \cD \right) = P\left( B_k \right) - 1_{A^c}.
\end{align*}
Moreover, we have that
\begin{align*}
    P(B_k) &= \frac{{p-s\choose \mt - 1}}{{p\choose \mt}} = \frac{\mt}{p} \frac{{p-s\choose \mt - 1}}{{p - 1\choose \mt -1}} = 
    \frac{\mt}{p} \frac{{p-\mt\choose s - 1}}{{p - 1\choose s - 1}}\\
    &=\prod_{i = 0}^{s-2}\left(\frac{p-\mt - i}{p - 1 - i}\right) \frac{\mt}{p} \geq \left( \frac{p - \mt - s + 2}{p - s + 1}\right)^{s-1} \frac{\mt}{p}\geq C_m^{s},
\end{align*}
where the second equality follows from the identity 
\begin{align*}
    \frac{ {n - h \choose k} }{ { n \choose k} } =  \frac{ {n - k \choose h} }{ { n \choose h} },
\end{align*}
with where $n = p - 1$, $h = s - 1$, and $k = \mt - 1$.
\end{proof}

\subsubsection{Combining results}
Our major result in Theorem \ref{theo:SappearNew} is formulated for the random (oracle) feature set $\cF = \cF(\cD, T, \cP)$. 
Note that this is an oracle feature set, as it depends on the true interactions $S_j$, which are not known in practice.
From the analysis in Section \ref{subsubsec:filtering} we know that we can obtain a consistent estimate of the oracle feature set $\cF$ by thresholding on MDI as in $\hat{\cF}_\epsilon$.
Recall that for a given $\epsilon$ the (random) set $\hat{\cF}_\epsilon$ can easily be obtained without any knowledge of the true model.
Based on Proposition \ref{prop:noisyCase}, we observe the following.

Recall that $\Omega_0$ is defined in \eqref{eq:Omega0}, $\cF$ is defined in \eqref{eq:defPF}, and $\hat \cF_\epsilon$ is defined in \eqref{eq:define_hat_cf} in the main text. We have the following theorem.
\begin{theorem}\label{theo:hatFequalF}
Under the assumption of Proposition \ref{prop:noisyCase} it holds true that for any fixed $\epsilon > 0$,
\begin{align}
&P_{(\cP, T)}\left( \Omega_0^c \;\middle|\; \cD \right)  \pto 0;\label{eq:omega}\\
& P_{(\cP, T)}\left( \hat{\cF}_\epsilon \nsubseteq {\cF} \;\middle|\; \cD \right) \pto 0;\label{eq:not_contain}\\
&P_{(\cP, T)}\left( \hat{\cF}_\epsilon \neq {\cF} \;\middle|\; \cD \right) \leq \left(\frac{ 4 \epsilon}{C_\beta^2 C_\gamma^{2\max_js_j -1}}\right)^{\tilde{C}} + \eta_n(\cD, \epsilon) \quad \text{ with }  \eta_n(\cD, \epsilon)  \pto 0;\label{eq:not_equal}
\end{align}
with $\tilde{C}$ as in Proposition \ref{prop:noisyCase}.
\end{theorem}
\begin{proof}
\eqref{eq:omega} follows directly from Proposition \ref{prop:noisyCase} ii) and the definition of $\Omega_0$ in \eqref{eq:Omega0}.

To prove \eqref{eq:not_contain}, one observes from Proposition \ref{prop:noisyCase} i) that for any $\epsilon > 0$, taking $\tilde \epsilon = \frac{ 4 \epsilon}{C_\beta^2 C_\gamma^{2\max_js_j -1}}$, the following happens with probability converging to one (as $n \to \infty$) 
\[
\max_{t\in T} \max_{k\in [p]/U(t)} \Delta_I^n (R_{t,l}(\gamma^*_{t,k};k), R_{t,r}(\gamma^*_{t,k};k)) < \frac{\tilde \epsilon}{4}C_\beta^2C_\gamma^{2\max_j s_j - 1} = \epsilon,
\]
which implies that $\hat \cF_\epsilon$ contains no irrelevant features. Thus,
\[
\liminf_{n \to \infty} P_{(\cD, T, \cP)} \left( \hat{\cF}_{\epsilon} \subseteq  \cF \right) = 1.
\]
Then by Markov inequality, we know $P_{(\cP, T)}\left( \hat{\cF}_\epsilon \nsubseteq {\cF} \;\middle|\; \cD \right)\pto 0$.

To prove \eqref{eq:not_equal}, we further note that by Proposition \ref{prop:noisyCase} iii),
\[
P\left(\min_{t\in \cP(t_{\leaf})} \min_{k\in U(t)} \Delta_I^n (R_{t,l}(\gamma^*_{t,k};k), R_{t,r}( \gamma^*_{t,k};k)) \geq \epsilon \Big| \cD
\right) \geq 
1 - \left(\frac{4\epsilon}{C_\beta^2C_\gamma^{2\max_js_j-1}}\right)^{\tilde{C}} - \eta_n(\cD, \epsilon).
\]
If \[\min_{t\in \cp(t_{\leaf})} \min_{k\in U(t)} \Delta_I^n (R_{t,l}(\gamma^*_{t,k};k), R_{t,r}( \gamma^*_{t,k};k)) \geq \epsilon\] and \[\max_{t\in T} \max_{k\in [p]/U(t)} \Delta_I^n (R_{t,l}(\gamma^*_{t,k};k), R_{t,r}(\gamma^*_{t,k};k)) < \epsilon,\] we know $\hat \cF_\epsilon = \cF$. Thus, we have
\begin{align}
    P_{(T, \cP)} \left( \hat{\cF}_{\epsilon} =  \cF \Big| \cD \right) &\geq 1 -  \left(\frac{ 4 \epsilon}{C_\beta^2 C_\gamma^{2\max_js_j -1}}\right)^{\tilde{C}} - \eta_n(\cD, \epsilon).
\end{align}
That completes the proof.
\end{proof}

Finally, we can combine Theorem \ref{theo:SappearNew}, Theorem \ref{theo:balancedRootFeature}, and Theorem \ref{theo:hatFequalF} to prove Theorem \ref{theo:mainResult_general_upper_bound} and \ref{theo:mainResult} in the main text.

\begin{proof}[Proof of Theorem \ref{theo:mainResult_general_upper_bound}]
Assume that $|S^\pm| = \tilde{s}$ and $S^\pm = \{(k_1, b_1), \ldots, (k_{ \tilde{s} }, b_{ \tilde{s}} ) \}$.
Analog as in Theorem \ref{theo:SappearNew}, for any feature $k \in [p]$, let $B^k$ be the Bernoulli random variable we draw when $k$ appears for the first time on $\cP$.
Recall the definition of $\hat{\cF}_\epsilon$, in particular, that $(k, b_k) \in \hat{\cF}_\epsilon$ only if $X_k$ appears the first time on $\cP$.
Thus, analog as for $\cF$ (recall the proof of Theorem \ref{theo:SappearNew}) we have that $(k, -1) \in \cF$ implies $B^k = -1$ and $(k, +1) \in \cF$ implies $B^k = +1$.
Thus,
\begin{align*}
    \{S^\pm \in \hat{\cF}_\epsilon \} \subset \{ B^{k_1} = b_1 \cap \ldots \cap  B^{k_{\tilde{s}}} = b_{\tilde{s}} \}
\end{align*}
and hence,
\begin{align*}
    \text{DWP}(S^\pm) = P_{(\cP, T)}(S^\pm \in \hat{\cF}_\epsilon | \cD ) \leq 
    P_{(\cP, T)}(B^{k_1} = b_1 \cap \ldots \cap B^{k_{\tilde{s}}} = b_{\tilde{s}} | \cD )
    = P_\cP( B^{k_1} = b_1 \cap \ldots B^{k_{\tilde{s}}} = b_{\tilde{s}} ) = 2^{- \tilde{s}}.
\end{align*}
\end{proof}

\begin{proof}[Proof of Theorem \ref{theo:mainResult}]
Assume that $|S^\pm| = \tilde{s}$ and $S^\pm = \{(k_1, b_1), \ldots, (k_{ \tilde{s} }, b_{ \tilde{s}} ) \}$ and let
\begin{align*}
    r_n(\cD, \epsilon) = \max\left( P_{(\cP,T)}(\Omega_0^c | \cD)  + \eta_n(\cD, \epsilon),\; 
    P_{(\cP, T)}( \hat{\cF}_\epsilon \nsubseteq {\cF} | \cD ) \right),
\end{align*}
with $\eta_n(\cD, \epsilon)$ as in Theorem \ref{theo:hatFequalF}.
It follows from Theorem \ref{theo:hatFequalF} that $r_n(\cD, \epsilon) \pto 0$ as $n \to \infty$.

\noindent \textbf{Proof of (Interaction lower bound):}\\
Assume that $S^\pm$ is a union interaction. Then we have that
\begin{align*}
    \text{DWP}(S^\pm) &= P_{(\cP, T)}(S^\pm \in \hat{\cF}_\epsilon | \cD )\\
    &\geq 
     P_{(\cP, T)}(S^\pm \in \cF | \cD ) - P_{(\cP, T)}( \hat{\cF}_\epsilon \neq {\cF} | \cD )\\
     &\geq P_{(\cP, T)}(S^\pm \in \cF | \cD ) - \left(\frac{ 4 \epsilon}{C_\beta^2 C_\gamma^{2\max_js_j -1}}\right)^{\tilde{C}} - \eta_n(\cD, \epsilon)\\
     &\geq 0.5^{\tilde{s}} - P_{(\cP,T)}(\Omega_0^c | \cD) -  \left(\frac{ 4 \epsilon}{C_\beta^2 C_\gamma^{2\max_js_j -1}}\right)^{\tilde{C}} - \eta_n(\cD, \epsilon)\\
     &\geq 0.5^{\tilde{s}} -  \left(\frac{ 4 \epsilon}{C_\beta^2 C_\gamma^{2\max_js_j -1}}\right)^{\tilde{C}} - r_n(\cD, \epsilon),
\end{align*}
where the second inequality follows from Corollary \ref{theo:hatFequalF} and the third inequality follows from Theorem \ref{theo:SappearNew}.

\noindent \textbf{Proof of (Non-interaction upper bound):}\\
Assume that $S^\pm$ is not a union interaction. Then we have that
\begin{align*}
    \text{DWP}(S^\pm) &= P_{(\cP, T)}(S^\pm \in \hat{\cF}_\epsilon | \cD )\\
    &\leq 
     P_{(\cP, T)}(S^\pm \in \cF | \cD ) + P_{(\cP, T)}( \hat{\cF}_\epsilon \nsubseteq {\cF} | \cD )\\
     &\leq  0.5^{\tilde{s}}(1 - C_{\mathrm{root}}(\cD) /2) + r_n(\cD, \epsilon),
\end{align*}
where the second inequality follows from Theorem \ref{theo:balancedRootFeature}.
\end{proof}

\FloatBarrier

\section{Additional figures}

\begin{figure}[h!]
    \centering
    \includegraphics[width=0.5\textwidth]{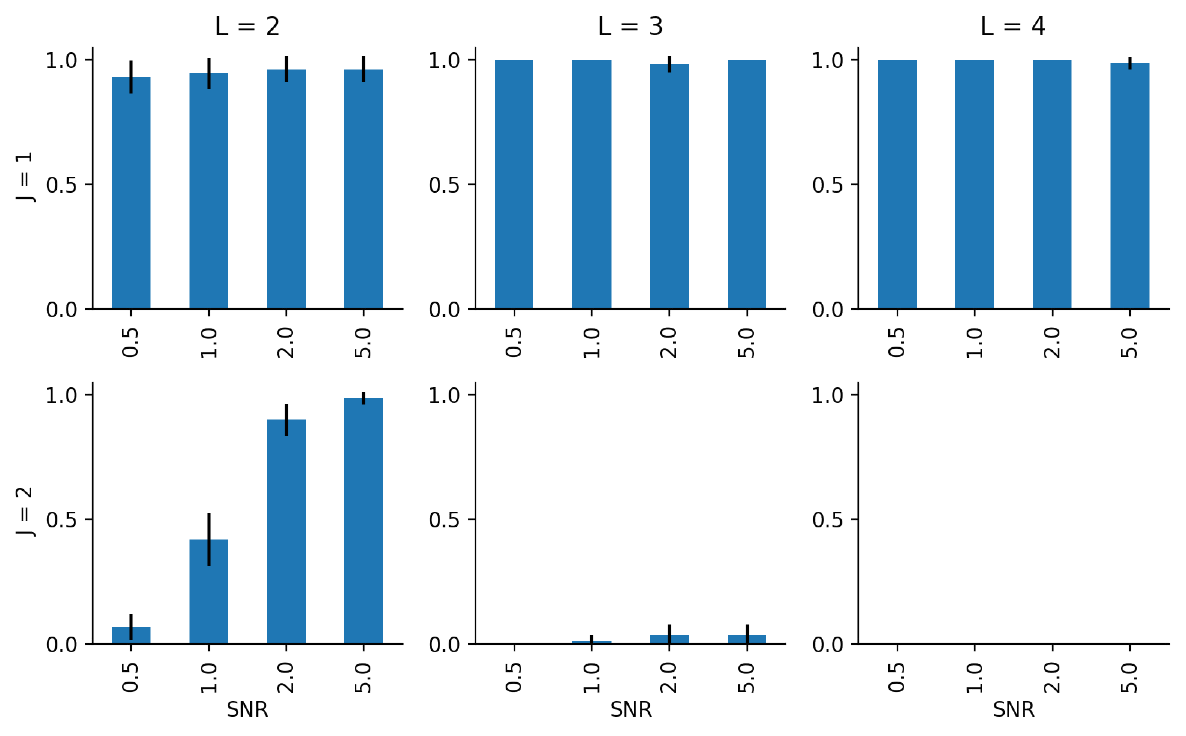}
    \caption{
    Simulation results for the performance of the LSSFind (Algorithm \ref{Algo:1}).
    The data is generated from an LSS model with Gaussian noise as in \eqref{eq:simulations_model_lss} with $n = 1,000$ samples and $p = 20$ features. Standard deviations of the proximity scores are given in error bars.
    Different number of basic interactions $K$ are shown in different rows, of interaction-orders $L$ in different columns, and of a series of SNRs on the x-axis. 
    The y-axis shows the proximity score in \eqref{eq:proxScore}.
    A proximity score of one corresponds to perfect recovery of all interactions simultaneously.
    }
    \label{fig:Sim1}
\end{figure}

\begin{figure}[h!]
\centering
\includegraphics[width=0.15\textwidth]{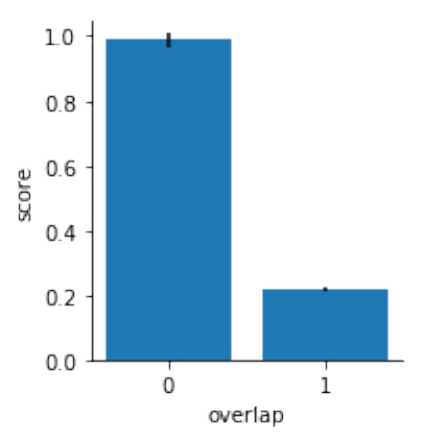}
\includegraphics[width=0.15\textwidth]{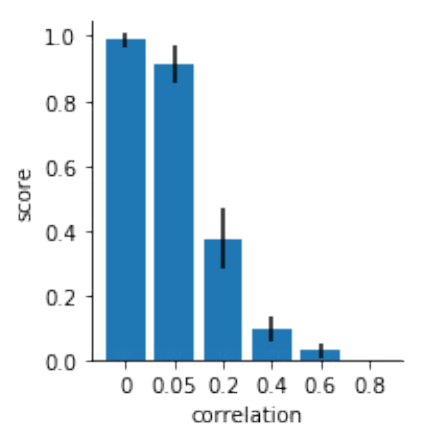}
\includegraphics[width=0.15\textwidth]{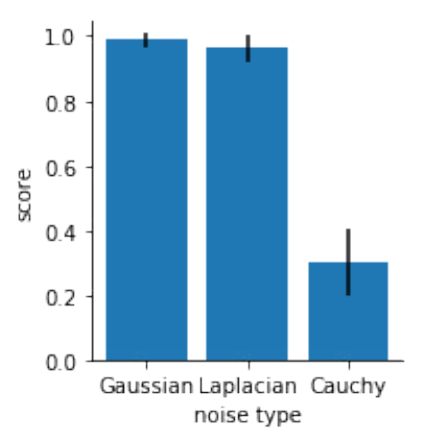}

\caption{
  Simulation results similar to those in Figure \ref{fig:Sim1} for interactions of order $L = 2$ and $J=2$, but when the data is generated from a mis-specified version of the LSS model, with $n = 1,000$ samples and $p = 20$ features.
  Left : signed features of different basic interactions are overlapping.
    When $\textrm{overlap}=1$, the basic interactions are $((1, -1), (2, -1))$, $( (2,-1)$, $(3, -1)$.
    Middle: different features are correlated instead of independent. 
    When $\text{corr}=\alpha$, the correlation between feature $j_1$ and $j_2$ is $\alpha^{|j_1 - j_2|}$.
    Right: the noise follows a Laplace or Cauchy distribution, instead of Gaussian distributions.
    }\label{fig:LSSviolation}
\end{figure}

\begin{figure}[h!]
\centering
\includegraphics[width=0.5\textwidth]{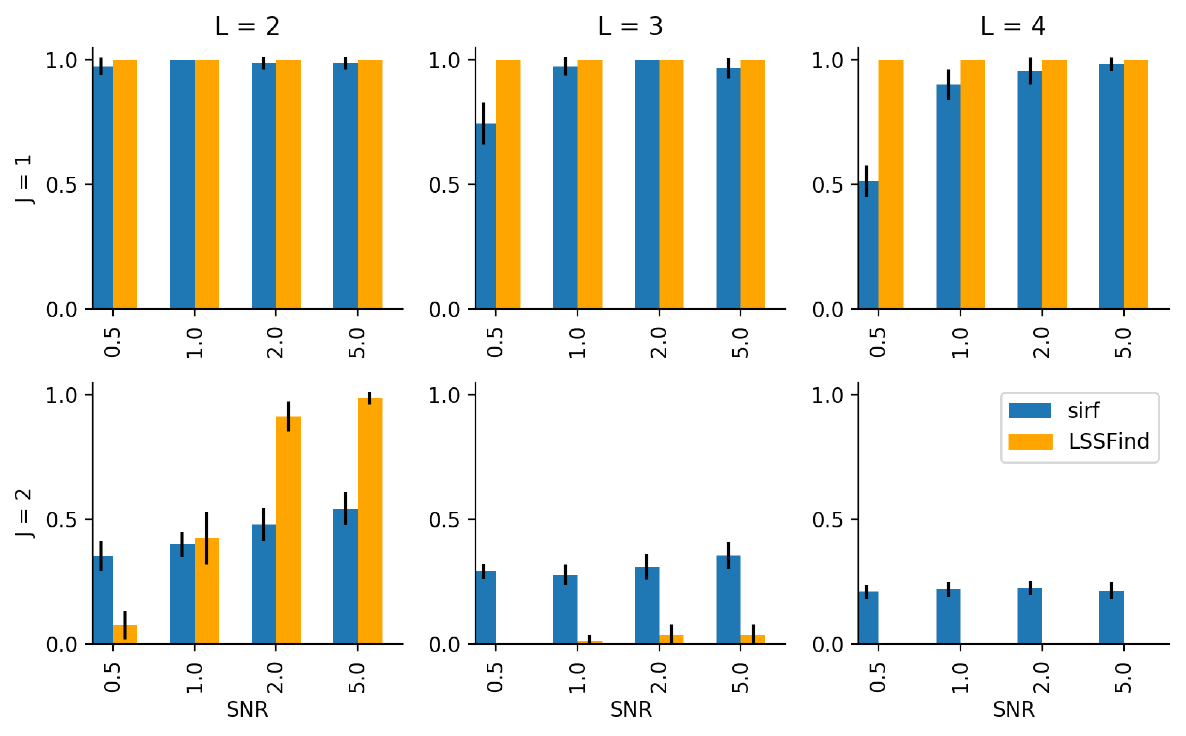}

\caption{
  Simulation results of LSSFind ({\color{orange} orange}) and iRF ({\color{blue} blue}) analog as in Figure \ref{fig:Sim1} but with the performance measure \eqref{eq:proxScore2} instead of \eqref{eq:proxScore}. 
    }\label{fig:LSSiRFcompare}
\end{figure}

\begin{figure}[h!]
\centering
\includegraphics[width=0.5\textwidth]{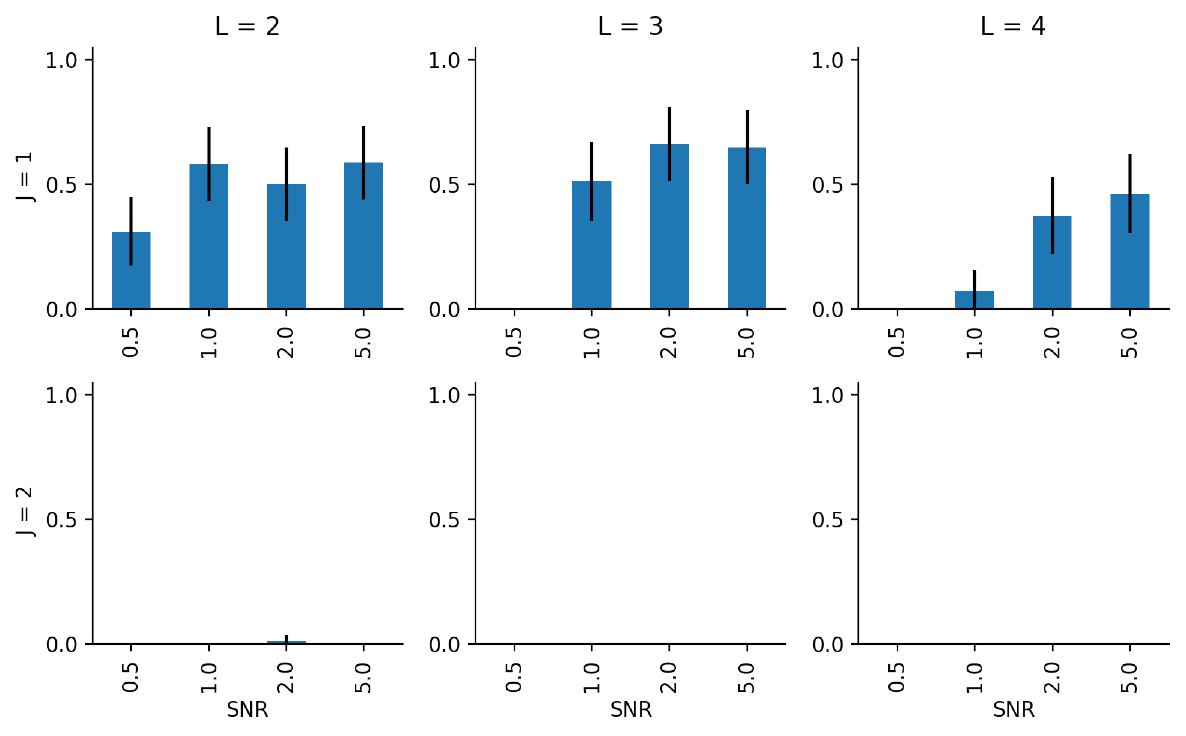}

\caption{
  Simulation results for iRF analog as in Figure \ref{fig:Sim1}. LSSFind has higher score when the number of basic interactions $K=1$ or when the order of interactions $L=2$. For other cases, neither methods have good performance. Note that when a different metric is used, the story is different, see Figure \ref{fig:LSSiRFcompare}. 
    }\label{fig:LSSiRFcompare2}
\end{figure}

\begin{figure}[h!]
\centering
\includegraphics[width=0.3\textwidth]{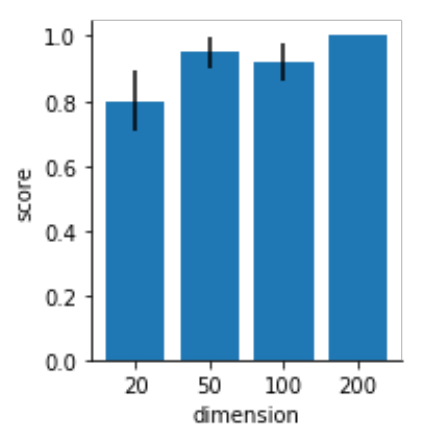}

\caption{
  Simulation results for LSSFind when the dimension $p$ grows and the samplesize grows at the rate of $log(p)$. LSSFind has higher score when the dimension grows higher.
    }\label{fig:high_dimension}
\end{figure}

\FloatBarrier

\section{Acknowledgments}
This work was supported in part by a Chan Zuckerberg Biohub Intercampus Research Award to BY.
MB was supported by Deutsche Forschungsgemeinschaft (DFG; German Research Foundation) Post-doctoral Fellowship BE 6805/1-1. 
MB acknowledges partial support from NSF Grant Big Data 60312.
BY acknowledges partial support from National Science Foundation grants NSF-DMS-1613002, 1953191, 2015341, and IIS 1741340.
This work was supported  in part by the Center for Science of Information (CSoI), an NSF Science and Technology Center, under grant agreement CCF-0939370, and by NSF and the Simons Foundation for the Collaboration on the Theoretical Foundations of Deep Learning through awards DMS-2031883 and \#814639 respectively.
Helpful comments of Sumanta Basu and Karl Kumbier are gratefully acknowledged.

\end{document}